\documentclass[11pt]{article}

\usepackage{amssymb, amsfonts, amsmath, amsthm} 

\newtheorem{theorem}{Theorem}[section]
\newtheorem{lemma}[theorem]{Lemma}
\newtheorem{proposition}[theorem]{Proposition}
\newtheorem{corollary}[theorem]{Corollary}
\newtheorem{prop-and-def}[theorem]{Proposition and Definition}

\theoremstyle{definition}
\newtheorem{definition}[theorem]{Definition}
\newtheorem{notation}[theorem]{Notation}
\newtheorem{remark}[theorem]{Remark}
\newtheorem{example}[theorem]{Example}

\newtheorem{Def-and-remark}[theorem]{Definition and Remark}
\newtheorem{def-and-remark}[theorem]{Notation and Remark}
\newtheorem{remark-and-def}[theorem]{Remark and Notation}
\newtheorem{ad-hoc-item}[theorem]{ }

\newcommand{\cA}{ {\cal A} }
\newcommand{\cP}{ {\cal P} }
\newcommand{\cPchi}{ \cP^{ ( \chi )} }
\newcommand{\cPchiopp}{ \cP^{ ( \chiopp )} }
\newcommand{\cT}{ {\cal T}_d }
\newcommand{\cX}{ {\cal X} }

\newcommand{\bC}{ {\mathbb C} }
\newcommand{\bN}{ {\mathbb N} }
\newcommand{\bZ}{ {\mathbb Z} }

\newcommand{\phivac}{ \varphi_{\mathrm{vac}} }
\newcommand{\xivac}{ \xi_{\mathrm{vac}} }
\newcommand{\piopp}{ \pi_{\mathrm{opp}} }
\newcommand{\chiopp}{  \chi_{ { }_{\mathrm{opp}} }  }
\newcommand{\omegaopp}{  \omega_{ { }_{\mathrm{opp}} }  }
\newcommand{\Luk}{ {\mathrm{Luk}} }
\newcommand{\term}{ {\mathrm{term}} }
\newcommand{\xtau}{\sigma}
\newcommand{\cstand}{ \rho_{ { }_{\lambda , \chi} } }
\newcommand{\leftstand}{ \rho_{ { }_{\lambda , \chi ; \ell} } }
\newcommand{\rightstand}{ \rho_{ { }_{\lambda , \chi ;r} } }

\newcommand{\arrowA}{ \stackrel{(1)}{\longrightarrow} }
\newcommand{\arrowB}{ \stackrel{(2)}{\longrightarrow} }
\newcommand{\arrowC}{ \stackrel{(3)}{\longrightarrow} }

\newcommand{\baone}{ \textcircled{1} }
\newcommand{\batwo}{ \textcircled{2} }
\newcommand{\bathr}{ \textcircled{3} }
\newcommand{\bafou}{ \textcircled{4} }
\newcommand{\bafiv}{ \textcircled{5} }
\newcommand{\basix}{ \textcircled{6} }
\newcommand{\basev}{ \textcircled{7} }
\newcommand{\baeig}{ \textcircled{8} }
\newcommand{\banin}{ \textcircled{9} }
\newcommand{\baa}{ \textcircled{a} }
\newcommand{\bab}{ \textcircled{b} }
\newcommand{\bac}{ \textcircled{c} }
\newcommand{\ban}{ \textcircled{n} }
\newcommand{\banminus}{ \textcircled{n'} }
\newcommand{\banminustwo}{ \textcircled{n''} }
\newcommand{\bass}{ \textcircled{s} }
\newcommand{\bax}{ \textcircled{x} }
\newcommand{\bayy}{ \textcircled{y} }

\begin{document}

$\ $

\begin{center}
{\bf\Large Double-ended queues and joint moments of}

\vspace{6pt}

{\bf\Large left-right canonical operators on full Fock space}

\vspace{20pt}

{\large Mitja Mastnak ${ }^{1}$\hspace{2cm}
Alexandru Nica \footnote{Research supported by a Discovery Grant
from NSERC, Canada.} }

\vspace{10pt}

\end{center}

\begin{abstract}

\noindent
We follow the guiding line offered by canonical operators on the 
full Fock space, in order to identify what kind of cumulant 
functionals should be considered for the concept of bi-free 
independence introduced in the recent work of Voiculescu.  By 
following this guiding line we arrive to consider, for a general 
noncommutative probability space $( \cA , \varphi )$, a family 
of ``$( \ell , r )$-cumulant functionals'' which enlarges the 
family of free cumulant functionals of the space.  In the 
motivating case of canonical operators on the full Fock space 
we find a simple formula for a relevant family of
$( \ell , r )$-cumulants of a $(2d)$-tuple
$(A_1, \ldots , A_d, B_1, \ldots , B_d)$, with
$A_1, \ldots , A_d$ canonical operators on the left 
and $B_1, \ldots , B_d$ canonical operators on the right.  This 
extends a known one-sided formula for free cumulants of 
$A_1, \ldots , A_d$, which establishes a basic operator model 
for the $R$-transform of free probability.
\end{abstract}

$\ $


\hspace{1.8cm}
bi-non-crossing partitions, bi-free cumulants.


\vspace{.5cm}

\begin{center}
{\bf\large 1. Introduction}
\end{center}
\setcounter{section}{1}

In this paper 
\footnote{This is the electronic version of an article published
in International Journal of Mathematics 26 (2015), Issue 02,
paper 1550016.
DOI: 10.1142/S0129167X15500160,
\copyright World Scientific Publishing Company.}
we follow the guiding line offered by a special 
type of canonical operators on the full Fock space, in order
to identify what kind of cumulant functionals should 
be considered for the concept of {\em bi-free independence}
introduced by D. Voiculescu \cite{V2013a}, \cite{V2013b}.

In Section 1.1 we give a general (rather informal) description of 
what is the above mentioned ``guiding line''; the main point of
the description is that we are upgrading from (I) to (II) in the 
diagrams displayed on the next page.

$\ $

{\bf 1.1 From (bi-)free independence to partitions, via 
canonical operators.}

The concept of free independence for noncommutative random 
variables was pinned down by D. Voiculescu in the 1980's 
(\cite{V1985}, \cite{V1986}), with inspiration from natural 
families of generators for algebras associated to free 
groups.  An important further idea brought by R. Speicher
\cite{S1994} was that calculations with freely independent 
random variables can be efficiently done by using the lattices 
$NC(n)$ of non-crossing partitions, and free cumulant functionals 
based on these lattices.  The main point of the cumulant approach 
is that free independence is equivalent to a condition (many times
easier to verify than the original definition) of ``vanishing of
mixed free cumulants'' for the random variables in question.

What is the guiding line which allows one to start from 
the original definition of free independence, based on free groups, 
and ``find'' the $NC(n)$'s?  Here we pursue (with due acknowledgement 
that there is more than one answer to the question just asked) the 
line provided by a special type of ``canonical'' operators on the 
full Fock space $\cT$ over $\bC^d$.  These are certain power series 
in creation/annihilation operators on $\cT$, which were invented in 
\cite{V1986} for $d=1$, then extended in \cite{N1996} to general $d$.
The fact that they play a role in free probability is not so
surprising, since $\cT$ itself is in a certain sense an incarnation 
of the free monoid on $d$ generators.  But now, the way how products 
of canonical operators act on the vacuum-vector of $\cT$ is encoded 
by the action of inserting and removing objects in a well-known 
gadget from theoretical computer science, called {\em lifo-stack} 
(lifo = abbreviation for ``last-in-first-out'').  Lifo-stacks are 
closely connected (through the intermediate of some lattice paths 
called Lukasiewicz paths) to non-crossing partitions, so overall 
we get a diagram like this:
\[
\mbox{(I)} \hspace{0.5cm}
\left(  \begin{array}{c}
        \mbox{free independence for}             \\
        \mbox{$d$ random variables}
        \end{array}  \right) \ \longrightarrow \ 
\left(  \begin{array}{c}
        \mbox{$d$-tuple of canonical}            \\
        \mbox{operators on full Fock space}
        \end{array}  \right) 
\]
\[
\mbox{ $\ $ \hspace{3cm} $\ $} \longrightarrow \
\left(  \begin{array}{c}
        \mbox{lifo-}        \\
        \mbox{stacks}
        \end{array}  \right) \ \longrightarrow \ 
\left(  \begin{array}{c}
        \mbox{non-crossing}    \\
        \mbox{partitions}
        \end{array}  \right) .
\]

In 2013, Voiculescu \cite{V2013a}, \cite{V2013b} started to
study the concept of {\em bi-free independence} for $d$ 
{\em pairs} of random variables in a noncommutative probability 
space.  This too is a concept inspired from looking at algebras 
associated to free groups, but its definition is in some sense 
indirect -- it ultimately boils down to a special way of 
representing the $2d$ random variables in question on a 
free-product space.  In order to advance the study of bi-free 
independence, it is of obvious importance (even more so than 
it was the case for plain free independence) to be able to 
re-phrase it as a vanishing condition on mixed cumulants, 
for a suitable construction of cumulant functionals.  The 
candidate for how to construct such ``bi-free cumulant functionals'' 
will come out of an upgrade of diagram (I), which tells us what 
lattices of partitions to use in the stead of $NC(n)$'s.  The 
upgrade is as follows:
\[
\mbox{(II)} \vspace{0.5cm}
\left(  \begin{array}{c}
        \mbox{bi-free independence for }                \\
        \mbox{$d$ pairs of random variables}
        \end{array}  \right) \ \arrowA \
\left(  \begin{array}{c}
        \mbox{$(2d)$-tuple of canonical}  \\
        \mbox{operators on full Fock space,} \\
        \mbox{$d$ on left and $d$ on right}
        \end{array}  \right) 
\]
\[
\mbox{$\ $ \hspace{2cm} $\ $} \arrowB \
\left(  \begin{array}{c}
        \mbox{double-ended}      \\
        \mbox{queues}
        \end{array}  \right) \ \arrowC \ 
\left(  \begin{array}{c}
        \mbox{bi-non-crossing}    \\
        \mbox{partitions}   
        \end{array}  \right) .
\]
(We have numbered the three arrows which appear in diagram (II),
in order to discuss them separately in the next subsection.)

The lattices of partitions obtained in (II) are indexed not 
only by a positive integer $n$, but also by a tuple 
$\chi = ( h_1, \ldots , h_n )$ where every $h_i$ is either 
the letter $\ell$ (for ``left'') or the letter $r$ (for 
``right'').  Throughout the paper, the notation used for 
such a lattice of partitions will be $\cPchi (n)$.  For 
every $n \in \bN$ and $\chi \in \{ \ell , r \}^n$ we have 
that $\cPchi (n)$ is a collection of partitions of
$\{ 1, \ldots , n \}$.  If $\chi = ( \ell, \ell , \ldots , \ell )$ 
or if $\chi = ( r,r, \ldots, r )$, then $\cPchi (n) = NC(n)$,
but for arbitrary $\chi \in \{ \ell , r \}^n$ we generally have 
$\cPchi (n) \neq NC(n)$.

$\ $

{\bf 1.2 Discussion of the three arrows in diagram (II).}

\vspace{6pt}

{\em Discussion of the connection ``$\arrowA$''.}
Let $d$ be a fixed positive integer, and let us also fix an 
orthonormal basis $e_1, \ldots , e_d$ for $\bC^d$.  Recall that 
the full Fock space over $\bC^d$ is
\begin{equation}   \label{eqn:12a}
\cT = \bC \oplus \bC^d \oplus ( \bC^d \otimes \bC^d )
\oplus \cdots \oplus ( \bC^d )^{\otimes n} \oplus \cdots
\end{equation}
In $\cT$ we have a preferred orthonormal basis, namely
\begin{equation}   \label{eqn:12b}
\{ \xivac \} \cup \{ e_{i_1} \otimes \cdots \otimes e_{i_n} \mid
n \geq 1, \, 1 \leq i_1, \ldots , i_n \leq d \} ,
\end{equation}
where $\xivac$ is a fixed unit vector in the first summand $\bC$
on the right-hand side of (\ref{eqn:12a}).  
For every $1 \leq i \leq d$ we denote by $L_i, R_i \in B( \cT )$ 
the left-creation and respectively the right-creation operator on 
$\cT$ defined by the vector $e_i$.  These are isometries which act 
on the preferred basis by $L_i ( \xivac ) = R_i ( \xivac ) = e_i$ 
and by
\[
L_i ( e_{i_1} \otimes \cdots \otimes e_{i_n} )
= e_i \otimes e_{i_1} \otimes \cdots \otimes e_{i_n}, \ \ 
R_i ( e_{i_1} \otimes \cdots \otimes e_{i_n} )
= e_{i_1} \otimes \cdots \otimes e_{i_n} \otimes e_i,
\]
for $n \geq 1$ and $1 \leq i_1, \ldots , i_n \leq d$.

Suppose we are given a polynomial $f$ without constant term 
in non-commuting indeterminates $z_1, \ldots , z_d$.  Write it 
explicitly as
\begin{equation}    \label{eqn:12c}
f(z_1, \ldots, z_d) 
= \sum_{n=1}^{\infty} \ \sum_{i_1, \ldots , i_n = 1}^d
\alpha_{( i_1, \ldots , i_n)} z_{i_1} \cdots z_{i_n},
\end{equation}
where the $\alpha$'s are in $\bC$ (and $\exists \, n_o \in \bN$ 
such that $\alpha_{( i_1, \ldots , i_n)} = 0$ for $n > n_o$).
For $1 \leq i \leq d$, let $A_i$ be the operator in $B ( \cT )$
defined as follows:
\begin{equation}    \label{eqn:12d}
A_i := L_i^{*} \Bigl( I 
+ \sum_{n=1}^{\infty} \ \sum_{i_1, \ldots , i_n = 1}^d
\alpha_{( i_1, \ldots , i_n)} L_{i_n} \cdots L_{i_1} \Bigr) .
\end{equation}
We will say that $(A_1, \ldots , A_d)$ is the
{\em $d$-tuple of left canonical operators} with symbol $f$.
The reason to care about this $d$-tuple is that it provides 
one of the possible approaches (historically the first) to 
the $R$-transform, a fundamental tool used in free probability.
More precisely: if we endow $B( \cT )$ with the vacuum-state 
$\phivac$ (i.e. with the vector-state defined by the vector
$\xivac$ from (\ref{eqn:12b})), then the $R$-transform of 
$(A_1, \ldots , A_d)$ with respect to $\phivac$ 
is 
\footnote{ So what we have here is a canonical construction which 
produces a $d$-tuple of operators with prescribed $R$-transform.
In this discussion $f$ could be a formal power 
series in $z_1, \ldots, z_d$, in which case $A_1, \ldots , A_d$
would live in a suitable algebra of formal operators on $\cT$, 
as described for instance on pp. 344-346 of \cite{NS2006}.  For 
the sake of not complicating the notations, we will stick to $f$ 
being a polynomial.}  
the 
given $f$.  For a detailed presentation of how this goes, see 
Lecture 21 of \cite{NS2006}.  Let's also note that, as explained 
in that lecture of \cite{NS2006} (Theorem 21.4 there), the 
derivation of the $R$-transform of $(A_1, \ldots , A_d )$ relies 
solely on the fact that $L_1, \ldots , L_d$ used in Equation 
(\ref{eqn:12d}) form a free family of Cuntz isometries.  The 
latter fact means, by definition, that the $L_i$'s are isometries 
with $L_i^{*} L_j = 0$ for $i \neq j$, and that one has
\[
\phivac (  L_{i_1} \cdots L_{i_m} 
           L_{j_1}^{*} \cdots L_{j_n}^{*}  ) = 0
\]
for all non-negative integers $m,n$ with $m+n \geq 1$ and all 
$i_1, \ldots , i_m, j_1, \ldots , j_n \in \{ 1, \ldots , d \}$.

Since the creation operators $R_1, \ldots , R_d$ on the right 
also form a free family of Cuntz isometries, it follows that
canonical $d$-tuples of operators could be equivalently constructed
by using creation and annihilation operators 
{\em on the right side} (instead of the left side favoured in
Equation (\ref{eqn:12d})).  For the present paper it is important 
to also write explicitly this second set of formulas.  So let $g$ 
be a polynomial without constant term in the same 
$z_1, \ldots , z_d$, and let us write it explicitly as
\begin{equation}    \label{eqn:12e}
g(z_1, \ldots, z_d) 
= \sum_{n=1}^{\infty} \ \sum_{i_1, \ldots , i_n = 1}^d
\beta_{( i_1, \ldots , i_n)} z_{i_1} \cdots z_{i_n}
\end{equation}
(with $\beta$'s in $\bC$, and where $\exists \, n_o$ such that
$\beta_{( i_1, \ldots , i_n)} = 0$ for $n > n_o$).
For $1 \leq i \leq d$, we put
\begin{equation}    \label{eqn:12f}
B_i := R_i^{*} \Bigl( I 
+ \sum_{n=1}^{\infty} \ \sum_{i_1, \ldots , i_n = 1}^d
\beta_{( i_1, \ldots , i_n)} R_{i_n} \cdots R_{i_1} \Bigr)
\in B( \cT ).
\end{equation}
Then $(B_1, \ldots , B_d)$ is called the
{\em $d$-tuple of right canonical operators} with symbol $g$,
and has the property that its $R$-transform with 
respect to $\phivac$ is the polynomial $g$ we started with.

If taken in isolation, the $B_i$'s from Equation (\ref{eqn:12f}) 
would merely duplicate what the $A_i$'s from (\ref{eqn:12d}) are
already doing.  What is interesting is to consider the 
{\em combined $(2d)$-tuple} of $A_i$'s and $B_i$'s -- this 
provides a significant example of $d$ pairs of left/right 
variables, as one wants to study in bi-free probability.

$\ $

{\em Discussion of the connection ``$\arrowB$''.}
We now examine the values of $\phivac$ on monomials made with 
the canonical operators $A_1, \ldots , A_d, B_1, \ldots , B_d$
that were considered above.  Let us first recap the one-sided 
case, where we look at an expectation
\[
\phivac (A_{j_1} \cdots A_{j_n})
= \langle A_{j_1} \cdots A_{j_n} \xivac \, , \, \xivac \rangle ,
\]
for some $n \geq 1$ and $1 \leq j_1, \ldots , j_n \leq d$.  If
we replace each of $A_{j_1},  \ldots , A_{j_n}$ by using 
(\ref{eqn:12d}), and if we expand the ensuing product of sums, 
then we arrive to act on $\xivac$ with products of the form
\begin{equation}    \label{eqn:12g}
L_{j_1}^{*} ( L_{i_{1,1}} \cdots L_{i_{1,m(1)}} ) \cdots
L_{j_n}^{*} ( L_{i_{n,1}} \cdots L_{i_{n,m(n)}} ),
\end{equation} 
where $m(1), \ldots , m(n) \geq 0$ are such that 
$m(1) + \cdots + m(n) = n$.
The action of such a product on $\xivac$ can be followed 
intuitively by thinking of how a collection of $n$ balls 
$\baone, \ldots , \ban$ moves through a lifo-stack --  the 
balls go into the stack 
in
\footnote{ The sizes of the groups of balls going in the stack
are $m(n), \ldots , m(2), m(1)$.  Here we ignore for the moment 
some relations that must also be imposed in between the indices 
$i_{1,1}, \ldots , i_{n,m(n)}$ and $j_1, \ldots , j_n$.}
groups 
(creation) and are taken out of the stack one by one (annihilation).  

What happens when we upgrade to a monomial which contains both some 
$A_i$'s and some $B_i$'s?  We can still proceed in the same way
as above, only that in (\ref{eqn:12g}) some of the $L_i$'s and 
$L_j^{*}$'s become $R_i$'s and $R_j^{*}$'s.  The intuition of $n$ 
balls moving through a device like a stack continues to work, but 
sometimes (when we have ``$R$'' instead of ``$L$'') the balls must 
go in or out ``by the other end of the stack''.  This is precisely 
the device which in theoretical computer science goes 
under the name of {\em double-ended queue}, or {\em deque} for 
short -- see e.g. Section 2.2 of \cite{K1973}.  Some pictures 
showing how we think about deques and how we use them in the 
present paper appear on pages 11-12 in Section 3.

$\ $

{\em Discussion of the connection ``$\arrowC$''.
The set of partitions $\cPchi (n)$.}
Once again, let us first recap the one-sided case, where we look at
$n$ balls moving through a lifo-stack.  Every possible scenario 
of how the balls move through the stack has associated to it a 
certain partition of $\{ 1, \ldots , n \}$, which we call 
``output-time partition'' (see Definition \ref{def-and-rem:3.3}, 
Example \ref{example:3.4}.2 below).  When we pursue the discussion 
started in (\ref{eqn:12g}), the concrete formula obtained for 
$\phivac (A_{j_1} \cdots A_{j_n})$ comes out as a sum indexed by 
all possible output-time partitions.  But the last-in-first-out 
rule of the stack forbids output-time partitions from having 
crossings!  What has come out is precisely a summation formula
over $NC(n)$ (as we knew it should).

What happens when we upgrade to the case of combined $A_i$'s and 
$B_i$'s?  We now have $n$ balls moving through a deque.  We still
have the concept of output-time partition associated to a scenario 
for how the balls move through the deque, but this partition may 
now have crossings, due to the interference between left and right 
moves (e.g. it is possible that some of the balls 
$\baone, \ldots , \ban$ enter the deque by one side and exit by 
the other).  

If we fix a tuple $\chi \in \{ \ell , r \}^n$, then the set 
$\cPchi (n)$ of bi-non-crossing partitions corresponding to $\chi$
will be defined in Section 3 of the paper as the set of all partitions 
of $\{ 1, \ldots , n \}$ which can arise as output-time partition 
for a deque-scenario compatible with $\chi$.  (See details in 
Definition \ref{def:3.5}.)  In the case when $\chi$ happens to be
either $( \ell , \ldots , \ell )$ or $( r, \ldots , r )$, then 
the deque is reduced to a lifo-stack, and $\cPchi (n)$ is equal 
to $NC(n)$.  For general $\chi \in \{ \ell, r \}^n$, it follows 
immediately from the definition that $\cPchi (n)$ has the same 
cardinality as $NC(n)$, but $\cPchi (n)$ is generally different 
from $NC(n)$ itself.

In Section 4 of the paper we will find (Theorem \ref{thm:4.9})
an alternative description for $\cPchi (n)$.  It says that
$\cPchi (n) = \{ \xtau_{ { }_{\chi} } \cdot \pi \mid 
\pi \in NC(n) \}$,
where $\xtau_{ { }_{\chi} }$ is a specific (concretely described)
permutation of $\{ 1, \ldots , n \}$ associated to $\chi$, and 
where the action of a permutation $\sigma$ on a partition $\pi$ is 
defined in the natural way (if $\pi = \{ V_1, \ldots , V_k \}$, then 
$\sigma \cdot \pi = \{ \sigma (V_1), \ldots , \sigma (V_k) \}$).  
This alternative description of $\cPchi (n)$ is very useful for 
concrete calculations.  It also gives immediately the fact that, 
with respect to reverse refinement order, $\cPchi (n)$ is a lattice 
isomorphic to $NC(n)$.

$\ $

{\bf 1.3 From partitions to cumulant functionals.}

% {\bf \boldmath{$( \ell, r )$}-cumulant functionals}
% Didactical note:  1.3 is a short subsection, and it is 
% natural to think about merging it with either 1.2 or 
% 1.4.  But this does not work, for the following reason: 
% both 1.2 and 1.4 go in the framework of the full Fock 
% space; it is important to make clear that the definition 
% of the $( \ell , r)$-cumulant functionals is not 
% restricted to that, but is rather a concept that makes
% sense in an arbitrary ncps.

In both classical and free probability theory,
the standard method to introduce cumulant functionals
goes by writing a so-called ``moment-cumulant'' formula,
where moments are expressed in terms of cumulants via 
summations over a suitable family of lattices of partitions.  
In particular, free cumulant functionals (as introduced by 
Speicher \cite{S1994}) have a moment-cumulant formula based 
on non-crossing partitions.  To be specific, let a 
noncommutative probability space $( \cA , \varphi )$ be 
given.  The free cumulant functionals of $( \cA ,\varphi )$ are
defined as the family of multilinear functionals
$( \, \kappa_n : \cA^n \to \bC \, )_{n=1}^{\infty}$ which 
is uniquely determined by the requirement that for every 
$n \geq 1$ and $a_1, \ldots , a_n \in \cA$ we have:
\[
\mbox{(M-FC)} \hspace{1cm}
\varphi (a_1 \cdots a_n ) =
\sum_{\pi \in NC (n)}  \, \Bigl( \
\prod_{V \in \pi}
\kappa_{|V|} 
( \, (a_1, \ldots , a_n) \mid V \, ) \ \Bigr).
\]

\vspace{6pt}

By the same token, the families of partitions $\cPchi (n)$ 
discussed at the end of section 1.2 can be used to define a 
concept of cumulant functionals, as follows.  Let a noncommutative
probability space $( \cA , \varphi )$ be given.  There exists a 
family of multilinear functionals
\[
\Bigl( \,  \kappa_{\chi} : \cA^n \to \bC \, \Bigr)_{ n \geq 1,
                                \,  \chi \in \{ \ell , r \}^n } 
\]
which is uniquely determined by the requirement that 
for every $n \geq 1$, $\chi \in \{ \ell , r \}^n$
and $a_1, \ldots , a_n \in \cA$ we have:
\[
\mbox{(M-F${ }_2$C)} \hspace{1cm}
\varphi (a_1 \cdots a_n ) =
\sum_{\pi \in \cPchi (n)}  \, \Bigl( \
\prod_{V \in \pi}
\kappa_{\chi \mid V} 
( \, (a_1, \ldots , a_n) \mid V \, ) \ \Bigr).
\]
The explanation of various notational details, the easy proof of 
existence/uniqueness, and a bit of further discussion around the 
functionals $\kappa_{\chi}$ are given in Section 5 of the paper.
We will refer 
to 
\footnote{ In view of the results obtained in \cite{CNS2014}, after 
the present paper was first circulated, it is justified to also 
refer to the $\kappa_{\chi}$'s as
``bi-free cumulant functionals'' associated to $( \cA , \varphi )$. }
these 
functionals as the {\em $( \ell , r )$-cumulant functionals} 
associated to the noncommutative probability space 
$( \cA , \varphi )$.  It is immediate that they provide an 
enlargement of the family of free cumulants 
$( \kappa_n )_{n=1}^{\infty}$ associated to the same space,
in the respect that
\[
\kappa_n 
= \kappa_{ ( \, \underbrace{\ell, \ldots , \ell}_n \, ) }
= \kappa_{ ( \, \underbrace{r, \ldots , r}_n \, ) },
\ \ n \geq 1.
\]

$\ $

{\bf 1.4 Back to canonical operators: main result of the paper.}

The concept of $( \ell ,r)$-cumulants was introduced in section 1.3 
for a general noncommutative probability space $( \cA, \varphi)$.  
Here we go back to the special space $( B( \cT ), \phivac )$ from
section 1.2.  Let $(A_1, \ldots A_d)$ and $(B_1, \ldots B_d)$ be 
$d$-tuples of left (respectively right) canonical operators with 
symbols $f$ and $g$, as in Equations (\ref{eqn:12d}) and 
(\ref{eqn:12f}).  We are interested in the $( \ell ,r)$-cumulants 
of the combined $(2d)$-tuple
$(A_1, \ldots , A_d, B_1, \ldots , B_d)$.  The point we want to 
establish is that a relevant family of such $( \ell,r)$-cumulants
simply consists of coefficients of either the polynomial $f$
(an $\alpha_{(i_1, \ldots ,i_n)}$ from Equation (\ref{eqn:12c}))
or the polynomial $g$ (a $\beta_{(i_1, \ldots ,i_n)}$ from 
Equation (\ref{eqn:12e})).  For the statement of our theorem it is
convenient to use a unified notation for $A_i$'s and $B_i$'s:
\[
A_i =: C_{i; \ell} \ \mbox{ and }
B_i =: C_{i; r}, \ \ \mbox{ for } 1 \leq i \leq d.
\]

$\ $

{\bf Theorem.} 
{\em Consider all the notations pertaining to 
$A_1, \ldots , A_d, B_1, \ldots , B_d$ that were introduced above.
Let $n$ be a positive integer and let 
$\chi = ( h_1, \ldots , h_n )$ be in $\{ \ell, r \}^n$. 
Let us record explicitly where are the occurrences of $\ell$ 
and of $r$ in $\chi$:
\[
\left\{   \begin{array}{ll}
\{ m \mid 1 \leq m \leq n, \, h_m = \ell \}
=: \{ m_{\ell} (1), \ldots , m_{\ell} (u) \} &  \mbox{with
                 $m_{\ell} (1) < \cdots < m_{\ell} (u)$,}       \\
                          &                                     \\
\{ m \mid 1 \leq m \leq n, \, h_m = r \}
=: \{ m_r (1), \ldots , m_r (v) \} &  \mbox{with
                              $m_r (1) < \cdots < m_r (v)$.}
\end{array}  \right.
\]
Then for every $i_1, \ldots , i_n \in \{ 1, \ldots , d \}$ 
we have
\begin{equation}   \label{eqn:14b}
\kappa_{\chi} ( C_{i_1; h_1}, \ldots , C_{i_n; h_n} ) = 
\left\{   \begin{array}{ll}
\alpha_{ ( i_{m_r (v)}, \ldots , i_{m_r (1)},
           i_{m_{\ell} (1)}, \ldots , i_{m_{\ell} (u)} ) }, 
         & \mbox{ if $h_n = \ell$, }                         \\
         &                                                   \\
\beta_{ ( i_{m_{\ell} (u)}, \ldots , i_{m_{\ell} (1)},
          i_{m_r (1)}, \ldots , i_{m_r (v)} ) },
         & \mbox{ if $h_n = r$. } 
\end{array}  \right.
\end{equation}
                                                            }
$\ $

{\bf Remark.}
The above theorem generalizes the result from 
\cite{N1996} which says that the canonical left $d$-tuple 
$(A_1, \ldots , A_d)$ has $R$-transform 
$R_{ (A_1, \ldots , A_d) } = f$.  Indeed, if the tuple $\chi$ 
from the theorem is set to be 
$( \ell, \ldots , \ell ) \in \{ \ell , r \}^n$, then 
Equation (\ref{eqn:14b}) says that
$\kappa_n ( A_{i_1}, \ldots , A_{i_n} ) 
= \alpha_{ (i_1, \ldots , i_n) }, \ \  \forall 
\, 1 \leq i_1, \ldots , i_n \leq d$.  So then
\begin{align*}
R_{ (A_1, \ldots , A_d) }  (z_1, \ldots , z_d)
& := \sum_{n=1}^{\infty} \, \sum_{i_1, \ldots , i_n =1}^d \,   
   \kappa_n ( A_{i_1}, \ldots , A_{i_n} ) z_{i_1} \cdots z_{i_n}  \\
& = \sum_{n=1}^{\infty} \, \sum_{i_1, \ldots , i_n =1}^d \,   
  \alpha_{ (i_1, \ldots , i_n) }  z_{i_1} \cdots z_{i_n}          \\
& = f(z_1, \ldots, z_d),
\end{align*}
as claimed.  (By setting $\chi = (r, \ldots ,r)$ we could,
of course, also infer from the above theorem that 
$R_{ (B_1, \ldots , B_d) } = g$.)

$\ $

{\bf Example.}
The upshot of the theorem is that every mixed moment of the 
$A_i$'s and $B_i$'s is written as a straight sum (no signs or 
coefficients!) indexed by $\cPchi (n)$, where every term of the 
sum is a product of $\alpha$'s and $\beta$'s.  For a concrete 
example, say that we are interested in the mixed moment 
$\phivac ( A_{i_1} B_{i_2} A_{i_3} B_{i_4} )$, for some given
$i_1, i_2, i_3, i_4 \in \{ 1, \ldots , d \}$.  In the above 
theorem we take $n=4$ and $\chi = ( \ell, r, \ell, r)$, and get:
\begin{align*}
\phivac ( A_{i_1} B_{i_2} A_{i_3} B_{i_4} ) 
& = \kappa_{( \ell , r, \ell , r)}
     ( A_{i_1}, B_{i_2}, A_{i_3}, B_{i_4} ) 
    + \kappa_{( \ell )} (A_{i_1})
      \, \kappa_{(r, \ell , r)} (B_{i_2}, A_{i_3}, B_{i_4})   \\
& \mbox{$\ $ \hspace{0.5cm} $\ $} + \cdots
    + \kappa_{\ell} (A_{i_1}) \, \kappa_r (B_{i_2}) \, 
      \kappa_{\ell} (A_{i_3}) \, \kappa_r (B_{i_4})          \\
& = \beta_{ (i_3, i_1, i_2, i_4) }                 
    + \alpha_{ (i_1) } \, \beta_{ (i_3, i_2, i_4) } + \cdots
    + \alpha_{ (i_1) } \, \beta_{ (i_2) } \, \alpha_{ (i_3) } 
                   \, \beta_{ (i_4) },
\end{align*}
a sum of $14$ terms corresponding to the $14$ partitions in
$\cP^{( \ell ,r, \ell , r)} (4)$.  (As is easily checked, 
$\cP^{( \ell ,r, \ell , r)} (4)$ is different from $NC(4)$ -- it 
contains the crossing partition
$\{ \, \{ 1,3 \}, \, \{ 2,4 \} \, \}$, and misses the 
non-crossing partition $\{ \, \{ 1,4 \}, \, \{ 2,3 \} \, \}$.)

Note that the nice feature of getting a summation formula without 
signs and coefficients is due precisely to the fact that we are 
using the tuple $\chi = ( \ell ,r, \ell, r)$ coming from the mixed 
moment that we want to calculate.  If we tried for instance to 
evaluate the same mixed moment via a sum indexed by 
$\cP^{( r,r,r,r)} (4) = NC(4)$ (which would just be the usual 
free cumulant expansion), then we would run from the very 
beginning into the term
\begin{align*}
\kappa_{(r,r,r,r)} ( A_{i_1}, B_{i_2}, A_{i_3}, B_{i_4} )
& = \kappa_4 ( A_{i_1}, B_{i_2}, A_{i_3}, B_{i_4} )           \\
& = \beta_{ (i_3, i_1, i_2, i_4) }
    + \alpha_{ (i_1, i_3) } \beta_{ (i_2, i_4) }
    - \alpha_{ (i_2, i_3) } \beta_{ (i_1, i_4) } ,
\end{align*}
and the nice structure in the formula for 
$\phivac ( A_{i_1} B_{i_2} A_{i_3} B_{i_4} )$ (a straight sum of 
products) would only emerge after going through some more 
complicated expressions, and doing cancellations between terms.

$\ $

{\bf 1.5 Vanishing mixed \boldmath{$( \ell , r )$}-cumulants, 
and a question.}

Suppose the polynomials $f$ and $g$ considered in section 1.2
are of the form
\begin{equation}   \label{eqn:15a}
\left\{   \begin{array}{l}
f( z_1, \ldots , z_d ) 
= f_1 (z_1) + \cdots + f_d (z_d) \ \ \mbox{ and }   \\
                                                    \\
g( z_1, \ldots , z_d ) 
= g_1 (z_1) + \cdots + g_d (z_d), 
\end{array}  \right.
\end{equation}
where $f_1, \ldots, f_d, g_1, \ldots , g_d$ are polynomials of 
one variable.  Then the formulas defining the canonical operators 
$A_1, \ldots , A_d, B_1, \ldots , B_d$ simplify to
\begin{equation}   \label{eqn:15b}
\left\{   \begin{array}{l}
A_i = L_i^{*} \, \bigl( I + f_i (L_i) \bigr),    \\
                                                 \\
B_i = R_i^{*} \, \bigl( I + g_i (R_i) \bigr), 
\end{array}  \right. 
\ \ 1 \leq i \leq d.
\end{equation}

The $d$ pairs of operators $(A_1, B_1), \ldots , (A_d, B_d)$ 
appearing in Equations (\ref{eqn:15b}) give an example of 
bi-free family of pairs of elements of a noncommutative 
probability space, in the sense of Voiculescu \cite{V2013a}.
On the other hand, in view of the formula for 
$( \ell , r )$-cumulants provided by Equation (\ref{eqn:14b}), 
the special case of $f,g$ considered in (\ref{eqn:15a}) can be 
equivalently described via the requirement that 
\[
\left\{   \begin{array}{l}
\kappa_{\chi} 
\bigl( C_{i_1; h_1}, \ldots , C_{i_n; h_n} \bigr) = 0   \\
                                                        \\
\mbox{ $\ $ whenever $\exists \, 1 \leq p < q \leq n$
       such that $i_p \neq i_q$. }
\end{array}  \right.
\]
This coincidence is in line with the fact that various forms 
of independence for noncommutative random variables which are 
considered in the literature have a combinatorial
incarnation expressed in terms of the vanishing of some mixed 
cumulants.  It is in fact tempting to make a definition and ask
a question, as follows.

$\ $

{\bf Definition.} Let $( \cA , \varphi )$ be a noncommutative 
probability space, and let

\noindent
$\bigl( \, \kappa_{\chi} : \cA^n \to \bC \, \bigr)_{ 
                         n \geq 1, \chi \in \{ \ell ,r \}^n }$
be the family of $( \ell , r )$-cumulant functionals of 
$( \cA , \varphi )$.
Let $a_1, b_1, \ldots , a_d, b_d$ be in $\cA$.  We say that
the pairs $(a_1, b_1), \ldots , (a_d, b_d)$ are 
{\em combinatorially-bi-free} to mean that the following 
condition is fulfilled: denoting $c_{i; \ell } := a_i$ and 
$c_{i;r} := b_i$ for $1 \leq i \leq d$, one has
\begin{equation}  \label{eqn:15c}
\left\{   \begin{array}{l}
\kappa_{\chi} 
\bigl( c_{i_1; h_1}, \ldots , c_{i_n; h_n} \bigr) = 0   \\
\mbox{ $\ $ whenever $n \geq 2$, 
       $i_1, \ldots , i_n \in \{ 1, \ldots , d \}$,
       $\chi = ( h_1, \ldots , h_n ) \in \{ \ell , r \}^n$ }  \\
\mbox{ $\ $ and $\exists \, 1 \leq p < q \leq n$
       such that $i_p \neq i_q$. }
\end{array}  \right.
\end{equation}

$\ $

$\ $

{\bf Question.}  Is it true that combinatorial-bi-freeness is 
equivalent to the (representation theoretic) 
concept of bi-freeness introduced by Voiculescu in \cite{V2013a}?

$\ $

$\ $

After the first version of the present paper was circulated, the 
above question was found to have an affirmative answer, in the 
paper \cite{CNS2014} by I. Charlesworth, B. Nelson and P. Skoufranis.

$\ $

$\ $

{\bf 1.6 Organization of the paper.}

Besides the present introduction, the paper has five other 
sections.  After some review of background in Section 2, we 
introduce the sets of partitions $\cPchi (n)$ in Section 3,
via the idea of examining double-ended queues.  The alternative
description of $\cPchi (n)$ via direct bijection with $NC(n)$
is presented in Section 4.  In Section 5 we introduce the family 
of $( \ell , r )$-cumulant functionals $\kappa_{\chi}$ that are 
associated to a noncommutative probability space.  Finally, in 
Section 6 we prove the main result of the paper, giving the 
formula for $( \ell , r )$-cumulants of canonical operators 
that was announced in Equation (\ref{eqn:14b}) above.  

$\ $

$\ $

\begin{center}
{\bf\large 2. Background on partitions and on Lukasiewicz paths}
\end{center}
\setcounter{section}{2}
\setcounter{equation}{0}
\setcounter{theorem}{0}

\begin{definition}   \label{def:2.1}
{\em [Partitions of $\{ 1, \ldots , n \}$.] }

\noindent
Let $n$ be a positive integer.  

\vspace{6pt}

$1^o$ We will let $\cP (n)$ denote the set of all partitions 
of $\{ 1, \ldots , n \}$.  A partition $\pi \in \cP (n)$ is 
thus a set $\pi = \{ V_1, \ldots , V_k \}$ where  
$V_1, \ldots , V_k$ (called the {\em blocks} of $\pi$) are
non-empty sets with $V_i \cap V_j = \emptyset$ for 
$i \neq j$ and with $\cup_{i=1}^k V_i = \{ 1, \ldots , n \}$. 

\vspace{6pt}

$2^o$ On $\cP (n)$ we consider the partial order given by
reverse refinement; that is, for $\pi, \rho \in \cP (n)$ we 
will write ``$\pi \leq \rho$'' to mean that for every block 
$V \in \pi$ there exists a block $W \in \rho$ such that 
$V \subseteq W$.
The minimal and maximal partition with respect to this partial 
order will be denoted as $0_n$ and $1_n$, respectively:
\begin{equation}  \label{eqn:21a}
0_n := \Bigl\{ \, \{ 1 \}, \ldots , \{ n \} \, \Bigr\} , 
\ \  1_n := \Bigl\{ \, \{ 1 , \ldots , n \} \, \Bigr\} .
\end{equation}

\vspace{6pt}

$3^o$ Let $\tau$ be a permutation of $\{ 1, \ldots , n \}$ and
let $\pi = \{ V_1, \ldots , V_k \}$ be in $\cP (n)$.  We will use 
the notation ``$\tau \cdot \pi$'' for the new partition 
$\tau \cdot \pi := 
\{ \, \tau (V_1) , \ldots , \tau (V_k) \, \} \in \cP (n)$.

\vspace{6pt}

$4^o$ Let $\pi$ be a partition in $\cP (n)$.  By the 
{\em opposite} of $\pi$ we will mean the partition 
\[
\piopp := \tau_o \cdot \pi \in \cP (n),
\]
where $\tau_o$ is the order-reversing permutation of 
$\{ 1, \ldots , n \}$ (with $\tau_o (m) = n+1-m$ for every
$1 \leq m \leq n$).

\vspace{6pt}

$5^o$  A partition $\pi \in \cP (n)$ is said to be 
{\em non-crossing} when it is not possible to find two distinct 
blocks $V, W \in \pi$ and numbers $a < b < c < d$ in
$\{ 1, \ldots , n \}$ such that $a,c \in V$ and $b,d \in W$.  
The set $NC(n)$ of all non-crossing partitions in $\cP (n)$ 
will play a significant role in this paper; for a review of basic 
some facts about it we refer to Lectures 9 
and 10 of \cite{NS2006}.  Let us record here that $NC(n)$ is one 
of the many combinatorial structures counted by Catalan numbers,
one has
$| \, NC(n) \, | = C_n := (2n)!/ n! (n+1)!$
(the $n$-th Catalan number).
\end{definition}

$\ $

\begin{definition}   \label{def:2.2} 
{\em [Lukasiewicz paths.] }

$1^o$
We will consider paths in $\bZ^2$ which start at $(0,0)$ and 
make steps of the form $(1,i)$ with $i \in \bN \cup \{ -1, 0 \}$.
Such a path,
\begin{equation}   \label{eqn:22a}
\lambda = \bigl( \, (0,0), 
(1,j_1), (2, j_2) , \ldots , (n,j_n) \, \bigr),
\end{equation} 
is said to be a {\em Lukasiewicz path} when it satisfies the 
conditions that $j_m \geq 0$ for every $1 \leq m \leq n$ and that
$j_n = 0$.

\vspace{6pt}

$2^o$ For every $n \geq 1$, we will use the notation $\Luk (n)$
for the set of all Lukasiewicz paths with $n$ steps.  For a path 
$\lambda \in \Luk (n)$ written as in Equation (\ref{eqn:22a}), 
we will refer to the vector
\begin{equation}   \label{eqn:22b}
\vec{\lambda} = ( j_1 - 0, j_2 -j_1, \ldots , j_n - j_{n-1} )
\in ( \bN \cup \{  -1, 0 \} )^n
\end{equation} 
as to the {\em rise-vector} of $\lambda$.  It is immediate how 
$\lambda$ can be retrieved from its rise-vector; moreover, it 
is immediate that a vector 
$(q_1, \ldots , q_n) \in ( \bN \cup \{  -1, 0 \} )^n$ appears 
as rise-vector $\vec{\lambda}$ for some $\lambda \in \Luk (n)$
if and only if its satisfies the conditions that 
\begin{equation}   \label{eqn:22c}
q_1 + \cdots + q_m \geq 0 
\mbox{ for every $1 \leq m \leq n$, and }
q_1 + \cdots + q_n = 0.
\end{equation}
\end{definition}

$\ $

$\ $

\begin{remark-and-def}   \label{rem-and-def:2.3}
{\em [The surjection $\Psi$ and the bijection $\Phi$.] }

\noindent
Let $n$ be a positive integer.

\vspace{6pt}

$1^o$ Let $\pi = \{ V_1, \ldots , V_k \}$ be a partition in 
$\cP (n)$.  Consider the vector 
$(q_1, \ldots , q_n) \in ( \, \bN \cup \{ -1, 0 \} \, )^n$ 
where for $1 \leq m \leq n$ we put
\begin{equation}   \label{eqn:23a}
q_m :=  \left\{  \begin{array}{ll}
|V_i| -1, & \mbox{ if $m = \min (V_i)$ for an 
                      $i \in \{ 1, \ldots , k \}$ },    \\
-1,       & \mbox{otherwise.}
\end{array}  \right.
\end{equation}
It is immediately seen that $(q_1, \ldots , q_n)$ satisfies 
the conditions listed in (\ref{eqn:22c}), hence it is the 
rise-vector of a uniquely determined path $\lambda \in \Luk (n)$.

\vspace{6pt}

$2^o$ We will denote by $\Psi : \cP (n) \to \Luk (n)$ (also 
denoted as $\Psi_n$, if needed to clarify what is $n$) the 
map which acts by
\[
\Psi ( \pi ) := \lambda, \ \ \pi \in \cP (n),
\]
with $\lambda$ obtained out of $\pi$ via the rise-vector 
described in (\ref{eqn:23a}).

\vspace{6pt}

$3^o$ The map $\Psi : \cP (n) \to \Luk (n)$ introduced above 
has the remarkable property that its restriction to $NC(n)$ 
gives a bijection from $NC(n)$ onto $\Luk (n)$; for the 
verification of this fact, see e.g. \cite{NS2006}, Proposition 
9.8.  We will denote by $\Phi : \Luk (n) \to NC(n)$ the inverse 
of this bijection.  That is: for every $\lambda \in \Luk (n)$,
we define $\Phi ( \lambda )$ to be the unique partition in 
$NC(n)$ which has the property that
\[
\Psi ( \, \Phi ( \lambda ) \, ) = \lambda .
\]
The bijection $\Phi$ confirms the well-known fact that
the set of paths $\Luk (n)$ has the same cardinality 
$C_n$ (Catalan number) as $NC(n)$.  
\end{remark-and-def}   

$\ $

$\ $

\begin{center}
{\bf\large 3. Double-ended queues and the sets of 
partitions \boldmath{$\cPchi (n)$} }
\end{center}
\setcounter{section}{3}
\setcounter{equation}{0}
\setcounter{theorem}{0}

\begin{Def-and-remark}  \label{def-and-rem:3.1}
{\em [Description of a deque device.] }

\noindent
We will work with a device called {\em double-ended queue}, 
or {\em deque} for short, which is used in the study of 
information structures in theoretical computer science (see 
e.g. \cite{K1973}, Section 2.2).  We will think about this 
device in the way depicted in Figures 1, 2 below, and described as 
follows.  Let $n$ be a fixed positive integer, and suppose we 
have $n$ labelled balls \baone, $\ldots ,$ \ban which have to 
move from an {\em input pipe} into an {\em output pipe} (both 
depicted vertically in the figures), by going through a 
{\em deque pipe} (depicted horizontally).  The deque device 
operates in discrete time: it goes through a sequence of states, 
recorded at times $t=0, \, t=1, \ldots , t=n$, where at time 
$t=0$ all the $n$ balls are in the input pipe and at $t=n$ they
are all in the output pipe.  Compared to the discussion in 
Knuth's treatise \cite{K1973}, we will limit the kinds of moves 
\footnote{ In \cite{K1973} the deque may perform any sequence of 
``insertions and deletions at either end of the queue'', where 
the term ``insertion'' designates the operation of moving some 
balls from the input pipe into the deque pipe, and ``deletion'' 
refers to moving some balls from the deque pipe into the output 
pipe.  In the present paper we will only allow the special moves
described in Definition \ref{def-and-rem:3.1}, which match, in 
some sense, the creation and annihilation performed by canonical 
operators on a full Fock space.}
that a deque can do, and we will require that:
\[
\left\{ \begin{array}{c}
\mbox{ For every $0 \leq i \leq n-1$, the deque
       device moves from its }                          \\
\mbox{ state at time $t=i$ to its state at 
       time $t=i+1$ by performing}                     \\
\mbox{ either a {\em ``left-$p$ move''} or a 
{\em ``right-$p$ move''}, where $p \in \bN \cup \{ 0 \}$. }
\end{array}  \right.
\]
The description of a left-$p$ move is like this:
\[
\left\{  \begin{array}{l}
\mbox{``Take the top $p$ balls that are in the input pipe, and 
        insert them }                                            \\
\mbox{into the deque pipe, from the {\em left}.  Then take the 
        {\em leftmost} ball in  }                                \\
\mbox{the deque pipe and insert it at the bottom of the 
        output pipe.''}
\end{array}  \right.
\]
The description of a right-$p$ move is analogous to the one of 
a left-$p$ move, only that the words ``left'' and ``leftmost'' 
get to be replaced by ``right'' and respectively ``rightmost''.

As is clear from the description of a left-$p$ (or right-$p$)
move, the number $p \in \bN \cup \{ 0 \}$ used in the move 
$(t=i) \rightsquigarrow (t=i+1)$ is subjected to the restriction 
that there exist $p$ balls (or more) in the input pipe at 
time $t=i$.  Note also the additional restriction that $p=0$
can be used in the move $(t=i) \rightsquigarrow (t=i+1)$ only if 
the device has at least 1 ball in the deque pipe at time $t=i$. 
\end{Def-and-remark} 

\vspace{3cm}

\begin{center}
\begin{tabular}{ccccccccccccccc}
 &        &  &  &  & $\vline$ &           & $\vline$ &  &  &  & 
                                                        &  &  &  \\
 & output & $\rightarrow$ &  &  & $\vline$ &  \batwo   
                                          & $\vline$ &  &  &  &  
                                                        &  &  &  \\
 &  pipe  &  &  &  & $\vline$ &  \bafiv   & $\vline$ &  &  &  &  
                                                        &  &  &  \\
 &        &  &  &  & $\vline$ &           & $\vline$ &  &  &  & 
                                                        &  &  &  \\
 &        &  &  &  &          &           &          &  &  &  &  
                                                        &  &  &  \\
 &        &  &  &  &          &           &          &  &  &  &  
                                                        &  &  &  \\
                                                           \hline
 &        &  &  &  &          &           &          &  &  &  & 
                                                        &  &  &  \\
 &        &  &  &  &  \baone  &  \bathr  
                                 &  \bafou  & & &  &  &  &  &    \\  
 &        &  &  &  &          &           &          &  &  &  &  
                                                        &  &  &  \\
                                                           \hline
 &        &  &  &  &          &           &          &  &  &  & 
                                                        &  &  &  \\
 &        & $\nearrow$ &  &   &           &          &  &  &  &  &  
                                                        &  &  &  \\
 & deque pipe &  &  &  &      &           &          &  &  &  &  
                                                        &  &  &  \\
 &        &  &  &  & $\vline$ &           & $\vline$ &  &  &  &  
                                                        &  &  &  \\
 &        &  &  &  & $\vline$ &  \basix   & $\vline$ &  &  &  & 
                                                        &  &  &  \\
 & input  & $\rightarrow$ &  &  & $\vline$ &  \basev  
                                         & $\vline$ &   &  &  &  
                                                        &  &  &  \\
 &  pipe  &  &  &  & $\vline$ &  \baeig  & $\vline$ &   &  &  &  
                                                        &  &  &  \\
 &        &  &  &  & $\vline$ &  \banin  & $\vline$ &   &  &  & 
                                                        &  &  &  \\
 &        &  &  &  & $\vline$ &           & $\vline$ &  &  &  & 
                                                        &  &  &  \\
 &        &  &  &  &          &           &          &  &  &  & 
                                                        &  &  & 
\end{tabular}

$\ $

$\ $

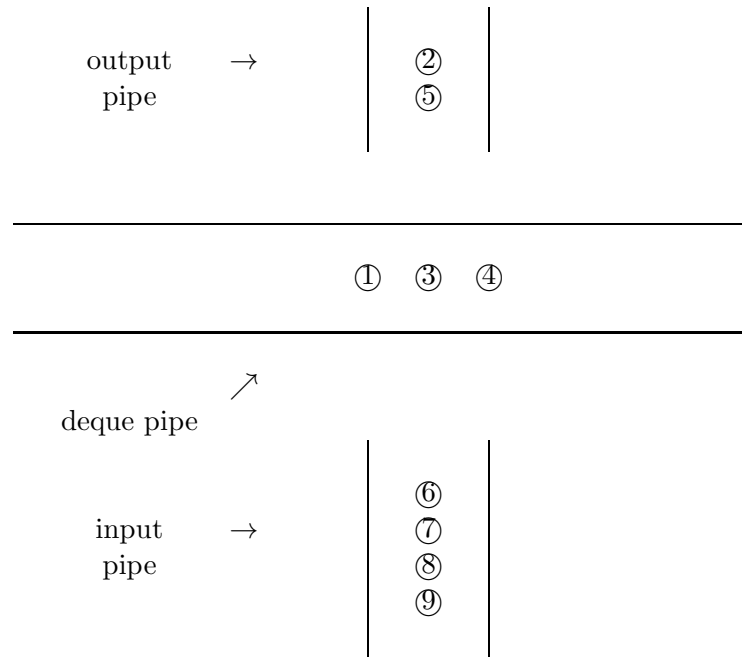
{\bf Figure 1:} {\em Say that $n=9$. Here is a possible state

of the deque device at time $t=2$.}
\end{center}

\newpage

\begin{center}
\begin{tabular}{cccccccccccccc}
 &    &  &  &  & $\vline$ &          & $\vline$ &  &  &  & 
                                                      &  &      \\
 &    &  &  &  & $\vline$ &  \batwo  & $\vline$ &  &  &  &  
                                                      &  &      \\
 &    &  &  &  & $\vline$ &  \bafiv   & $\vline$ &  &  &  &  
                                                       &  &     \\
 &    &  &  &  & $\vline$ &  \baeig   & $\vline$ &  &  &  &  
                                                       &  &     \\
 &    &  &  &  & $\vline$ &           & $\vline$ &  &  &  & 
                                                       &  &     \\
 &    &  &  &  &          &           &          &  &  &  &  
                                                       &  &     \\
 &    &  &  &  &          &           &          &  &  &  &  
                                                       &  &     \\
                                                           \hline
 &    &  &  &  &          &           &          &  &  &  & 
                                                       &  &     \\
 &    &  &  &  \basev     &  \basix   &  \baone  
            &  \bathr     &  \bafou   &          &  &  &  &     \\  
 &    &  &  &  &          &           &          &  &  &  &  
                                                       &  &     \\
                                                           \hline
 &    &  &  &  &          &           &          &  &  &  & 
                                                       &  &     \\
 &    &  &  &  &          &           &          &  &  &  &  
                                                       &  &     \\
 &    &  &  &  &          &           &          &  &  &  &  
                                                       &  &     \\
 &    &  &  &  & $\vline$ &           & $\vline$ &  &  &  &  
                                                       &  &     \\
 &    &  &  &  & $\vline$ &           & $\vline$ &  &  &  & 
                                                       &  &     \\
 &    &  &  &  & $\vline$ &  \banin   & $\vline$ &  &  &  &  
                                                       &  &     \\
 &    &  &  &  & $\vline$ &           & $\vline$ &  &  &  &  
                                                       &  &     \\
 &    &  &  &  & $\vline$ &           & $\vline$ &  &  &  & 
                                                       &  &     \\
 &    &  &  &  &          &           &          &  &  &  & 
                                                       &  &    
\end{tabular}

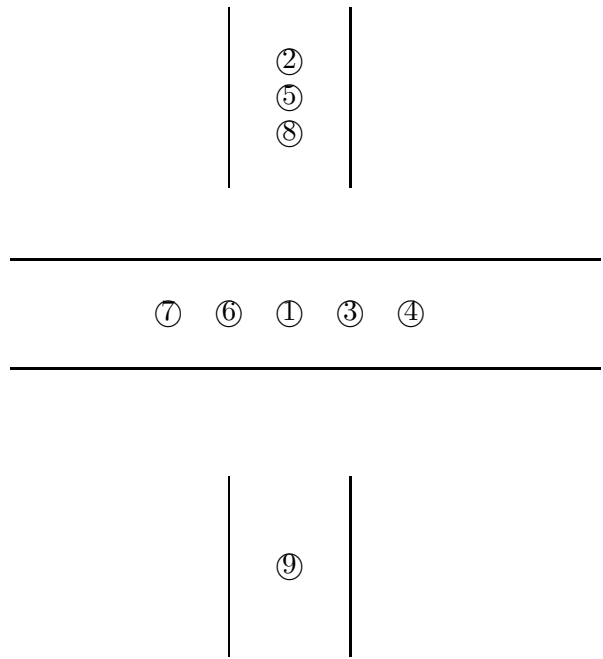
{\bf Figure 2:} {\em The deque device from Figure 1, at time 
$t=3$,

after performing a left-$3$ move.}

\vspace{0.45cm}

\begin{center}
******************************************************************
\end{center}

\vspace{0.45cm}

\begin{tabular}{cccccccccccccc}
 &    &  &  &  & $\vline$ &          & $\vline$ &  &  &  & 
                                                      &  &      \\
 &    &  &  &  & $\vline$ &  \batwo  & $\vline$ &  &  &  &  
                                                      &  &      \\
 &    &  &  &  & $\vline$ &  \bafiv   & $\vline$ &  &  &  &  
                                                       &  &     \\
 &    &  &  &  & $\vline$ &  \bafou   & $\vline$ &  &  &  &  
                                                       &  &     \\
 &    &  &  &  & $\vline$ &           & $\vline$ &  &  &  & 
                                                       &  &     \\
 &    &  &  &  &          &           &          &  &  &  &  
                                                       &  &     \\
 &    &  &  &  &          &           &          &  &  &  &  
                                                       &  &     \\
                                                           \hline
 &    &  &  &  &          &           &          &  &  &  & 
                                                       &  &     \\
 &    &  &  &  &  \baone  &  \bathr   &          &  &  &  &  
                                                       &  &       \\  
 &    &  &  &  &          &           &          &  &  &  &  
                                                       &  &     \\
                                                           \hline
 &    &  &  &  &          &           &          &  &  &  & 
                                                       &  &     \\
 &    &  &  &  &          &           &          &  &  &  &  
                                                       &  &     \\
 &    &  &  &  &          &           &          &  &  &  &  
                                                       &  &     \\
 &    &  &  &  & $\vline$ &           & $\vline$ &  &  &  &  
                                                       &  &     \\
 &    &  &  &  & $\vline$ &  \basix   & $\vline$ &  &  &  & 
                                                       &  &     \\
 &    &  &  &  & $\vline$ &  \basev   & $\vline$ &  &  &  &  
                                                       &  &     \\
 &    &  &  &  & $\vline$ &  \baeig   & $\vline$ &  &  &  &  
                                                       &  &     \\
 &    &  &  &  & $\vline$ &  \banin   & $\vline$ &  &  &  &  
                                                       &  &     \\
 &    &  &  &  & $\vline$ &           & $\vline$ &  &  &  & 
                                                       &  &     \\
 &    &  &  &  &          &           &          &  &  &  & 
                                                       &  &    
\end{tabular}

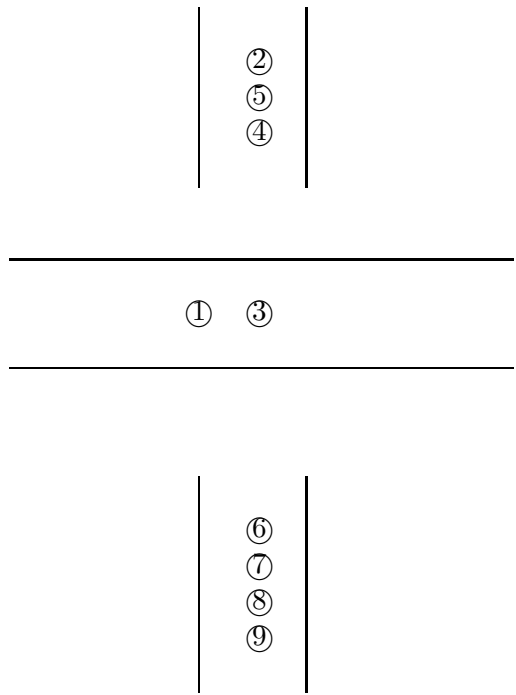
{\bf Figure 3:} {\em The deque device from Figure 1, at time 
$t=3$,

after performing a right-$0$ move.}
\end{center}

\vspace{2.5cm}

\begin{Def-and-remark}   \label{rem:3.2}
{\em [Deque-scenarios.] }

\noindent
Let $n$ be the same fixed positive integer as in Definition 
\ref{def-and-rem:3.1}, and consider the deque device described
there.  We assume that at time $t=0$ the $n$ balls are sitting 
in the input pipe in the order \baone, $\ldots ,$ \ban, counting 
top-down.  From the description of the moves of the device,
it is clear that for every $0 \leq i \leq n$ there are exactly
$i$ balls in the output pipe at time $t=i$.  In particular, 
all $n$ balls find themselves in the output pipe at time $t=n$ 
(even though they may not be sitting in the same order 
as at time $t=0$).

We will use the name {\em deque-scenario} to refer to a possible 
way of moving the $n$ balls through the deque device, according 
to the rules described above.  Every deque-scenario is thus 
determined by an array of the form
\begin{equation}   \label{eqn:32a}
\left(   \begin{array}{ccc}
p_1 &  \cdots & p_n    \\
h_1 &  \cdots & h_n 
\end{array}  \right) , \ \ 
\end{equation}
with $p_1, \ldots , p_n \in \bN \cup \{ 0 \}$ and 
$h_1, \ldots , h_n \in \{ \ell , r \}$; this array simply records 
the fact that in order to go from its state at time $t=i-1$ to 
its state at time $t=i$, the device has executed
\[
\left\{  \begin{array}{ll}
\mbox{a left-$p_i$ move,}  & \mbox{ if $h_i = \ell$  }   \\
                           &                             \\
\mbox{a right-$p_i$ move,} & \mbox{ if $h_i = r$ , } 
\end{array}  \right.   \ \ 1 \leq i \leq n.
\]

Let us observe that the top line of the array in (\ref{eqn:32a}) 
must satisfy the inequalities
\begin{equation}   \label{eqn:32b}
\left\{  \begin{array}{l}
p_1 + \cdots + p_i \geq i, \ \ \forall \, 1 \leq i \leq n,  \\
                                                            \\
\mbox{where for $i=n$ we must have } p_1 + \cdots + p_n =n.
\end{array}  \right.
\end{equation}
This is easily seen by counting that at time $t=i$ there 
are $i$ balls in the output pipe and $n- (p_1 + \cdots + p_i)$
balls in the input pipe, which leaves a difference of 
\[
n - \Bigl( i + n - (p_1 + \cdots + p_i) \Bigr) 
= (p_1 + \cdots + p_i) - i
\]
balls that must be in deque pipe. (And of course, the number of 
balls found in the deque pipe at $t=i$ must be $\geq 0$, with 
equality for $t=n$.)

It is easy to see that conversely, every array as in (\ref{eqn:32a})
with $p_1, \ldots , p_n$ satisfying (\ref{eqn:32b}) will define 
a working deque-scenario -- the inequalities 
$p_1 + \cdots + p_i \geq i$ ensure that we never run into the 
situation of having to ``move a ball out of the empty deque pipe''.

Thus, as a mathematical object, the set of deque-scenarios can 
be simply introduced as the set of arrays of the kind shown in
(\ref{eqn:32a}), and where (\ref{eqn:32b}) is satisfied.  

Moreover, let us observe that condition (\ref{eqn:32b}) can be 
read as saying that the $n$-tuple 
\begin{equation}   \label{eqn:32c}
( p_1 -1 , \ldots , p_n -1 ) \in 
\bigl( \bN \cup \{ -1, 0 \} \bigr)^n
\end{equation}  
is the rise-vector of a uniquely determined Lukasiewicz path 
$\lambda$, as reviewed in Section 2.  We will then refer to the 
deque-scenario described by the array (\ref{eqn:32a}) as the
{\em deque-scenario determined by $( \lambda , \chi )$}, where 
$\lambda \in \Luk (n)$ has rise-vector given by (\ref{eqn:32c})
and $\chi$ is the $n$-tuple 
$( h_1 , \ldots , h_n ) \in \{ \ell , r \}^n$ from the second 
line of (\ref{eqn:32a}).
\end{Def-and-remark}

$\ $

$\ $

\begin{Def-and-remark}    \label{def-and-rem:3.3}
{\em [Output-time partition associated to a deque-scenario.] }

\noindent
We consider the same notations as above and we look at the 
deque-scenario determined by $( \lambda , \chi )$, where 
$\lambda \in \Luk (n)$ has rise-vector 
$\vec{\lambda} = ( p_1 -1 , \ldots , p_n -1 )$ and where
$\chi = (h_1, \ldots , h_n) \in \{ \ell , r \}^n$.

Let $i \in \{ 1, \ldots , n \}$ be such that $p_i > 0$, and 
consider the $i$-th move of the deque device (the move that takes
the device from its state at $t=i-1$ to its state at $t=i$).  
In that move there is a group of $p_i$ balls (namely those with 
labels from $p_1 + \ldots + p_{i-1} +1$ to $p_1 + \ldots + p_i$)
which leave together the input pipe.  These balls arrive in 
the output pipe one by one, at various later times, which we
record as 
\begin{equation}  \label{eqn:33a}
t^{(i)}_1 < t^{(i)}_2 < \cdots < t^{(i)}_{p_i} .
\end{equation}
Observe that in particular we have $t^{(i)}_1 = i$; indeed, it is 
also part of the $i$-th move of the device that the ball with label
$p_1 + \cdots + p_i$ goes from the deque pipe into the output pipe.  
Let us make the notation
\[
T_i := \{ t^{(i)}_1, \ldots , t^{(i)}_{p_i} \}, 
\]
where $t^{(i)}_1, \ldots , t^{(i)}_{p_i}$ are from  (\ref{eqn:33a}). 

In the preceding paragraph we have thus constructed a set 
$T_i \subseteq \{ 1, \ldots , n \}$ for every $1 \leq i \leq n$
such that $p_i > 0$.  It is clear from the construction that 
for every such $i$ we have 
\begin{equation}  \label{eqn:33b}
| T_i | = p_i \mbox{ and } \min (T_i ) = i.
\end{equation}
It is also clear that the sets 
\begin{equation}  \label{eqn:33c}
\{ T_i \mid 1 \leq i \leq n \mbox{ such that } p_i > 0 \}
\end{equation}
form together a partition of $\pi \in \cP (n)$.  We will refer to
this $\pi$ as the {\em output-time partition} associated to the 
pair $( \lambda , \chi )$.
\end{Def-and-remark}

$\ $

$\ $

\begin{example}    \label{example:3.4}

$1^o$ A concrete example: say that $n=5$, that 
$\lambda \in \Luk (5)$ has rise-vector
$\vec{\lambda} = (2, -1, 1, -1 , -1)$, and that 
$\chi = ( r, \ell , \ell , r, \ell )$.
In the deque-scenario associated to this pair $( \lambda , \chi )$,
there are two groups of balls that are moved from the input pipe 
into the deque pipe: first group consists of
\baone, \batwo, \bathr,
which arrive in the deque pipe at time $t=1$; the second group 
consists of \bafou, \bafiv, which arrive in the 
deque pipe at time $t=3$.  The final order of the balls in the 
output pipe (counting downwards) is 
\begin{center}
\bathr,   \baone,   \bafiv,  \batwo,  \bafou,
\end{center}
and the output-time partition associated to $( \lambda , \chi )$ 
is $\pi = \{ \, \{ 1,2,4 \} , \, \{ 3, 5 \}  \, \} \in \cP (5)$.

\vspace{6pt}

$2^o$ Let $n$ be a positive integer, and consider the $n$-tuple 
$\chi_{ { }_{\ell} } := 
( \ell, \ldots , \ell ) \in \{ \ell , r \}^n$.
For any $\lambda \in \Luk (n)$, the deque-scenario determined by
$\lambda$ and $\chi_{ { }_{\ell} }$ is what one might call a 
``lifo-stack process'' (where lifo is a commonly used abbreviation 
for last-in-first-out).  It is easy to see that the output-time
partition associated to the pair $( \lambda , \chi_{ { }_{\ell} } )$ 
is the non-crossing partition $\Phi ( \lambda )$, where 
$\Phi : \Luk (n) \to \cP (n)$ is as reviewed in Remark 
\ref{rem-and-def:2.3}.3.

A similar statement holds if instead of $\chi_{ { }_{\ell} }$ 
we use the $n$-tuple $\chi_{ { }_r } := (r, \ldots ,r)$;
that is, the output-time partition associated to 
$( \lambda , \chi_{ { }_r })$ is the same 
$\Phi ( \lambda ) \in NC(n)$ as above.
\end{example}

$\ $

$\ $

\begin{definition}   \label{def:3.5}

$1^o$  Consider a pair $( \lambda , \chi )$ where 
$\lambda \in \Luk (n)$ and $\chi \in \{ \ell , r \}^n$, and 
let $\pi \in \cP (n)$ be the output-time partition associated
to $( \lambda , \chi )$ in Definition \ref{def-and-rem:3.3}.
We will denote this partition $\pi$ as 
``$\Phi_{\chi} ( \lambda )$''.

\vspace{6pt}

$2^o$  Let $\chi$ be an $n$-tuple in $\{ \ell , r \}^n$.  The 
notation introduced in $1^o$ above defines a function 
$\Phi_{\chi} : \Luk (n) \to \cP (n)$.  We define
\begin{equation}   \label{eqn:35a}
\cPchi (n) := \{  \Phi_{\chi} ( \lambda ) \mid 
\lambda \in \Luk (n) \} \subseteq \cP (n).
\end{equation}
\end{definition}

$\ $

\begin{proposition}   \label{prop:3.6}
Let $n$ be a positive integer, let $\chi$ be an $n$-tuple 
in $\{ \ell , r \}^n$, and consider the function 
$\Phi_{\chi} : \Luk (n) \to \cP (n)$ introduced in 
Definition \ref{def:3.5}.  

$1^o$ $\Phi_{\chi}$ is injective, hence it gives a bijection 
between $\Luk (n)$ and $\cPchi (n)$.

$2^o$ Let $\Psi_{\chi} : \cPchi (n) \to \Luk (n)$ be the 
function inverse to $\Phi_{\chi}$.  Then $\Psi_{\chi}$ is the 
restriction to $\cPchi (n)$ of the canonical surjection 
$\Psi : \cP (n) \to \Luk (n)$ that was reviewed in Remark 
\ref{rem-and-def:2.3}.2.
\end{proposition}

\begin{proof}  
Both parts of the proposition will follow if we can prove that 
$\Psi \circ \Phi_{\chi}$ is the identity map on $\Luk (n)$.  Thus
given a path $\lambda \in \Luk (n)$ and denoting 
$\Phi_{\chi} ( \lambda ) =: \pi$, we have to show that 
$\Psi ( \pi ) = \lambda$.  But the latter fact is clear from the 
observation made in (\ref{eqn:33b}) of Remark \ref{def-and-rem:3.3}.
\end{proof}

$\ $

\begin{remark}   \label{rem:3.7}
Let $n$ be a positive integer.

$1^o$ From Proposition \ref{prop:3.6} and the fact that 
$| \Luk (n) | = C_n$ ($n$-th Catalan number), it follows 
that $| \cPchi (n) | = C_n$ for every 
$\chi \in \{ \ell , r \}^n$.

\vspace{6pt}

$2^o$ Suppose that $\chi = ( \ell, \ldots , \ell )$.
The discussion from Example \ref{example:3.4}.2 shows that in 
this case we have $\cPchi (n) = NC(n)$.  Similarly, we also have
$\cPchi (n) = NC(n)$ in the case when 
$\chi = (r, \ldots , r)$.

\vspace{6pt}

$3^o$ If $n \leq 3$, then it is clear from cardinality 
considerations that $\cPchi (n) = \cP (n) = NC(n)$, no matter 
what $\chi \in \{ \ell , r \}^n$ we consider.  

For $n \geq 4$, cardinality considerations now show that 
$\cPchi (n)$ is a proper subset of $\cP (n)$.  It is usually 
different from $NC(n)$.  (For instance the output-time partition 
from Example \ref{example:3.4}.1 is not in $NC(5)$, showing that  
$\chi = ( r, \ell, \ell , r, \ell ) \in \{ \ell , r \}^5$ has 
$\cPchi (5) \neq NC(5)$.)
Some general properties of the sets of partitions $\cPchi (n)$ 
will follow from their alternative description provided in the 
next section.
\end{remark}

$\ $

$\ $

\begin{center}
{\bf\large 4. An alternative description 
              for \boldmath{$\cPchi (n)$} }
\end{center}
\setcounter{section}{4}
\setcounter{equation}{0}
\setcounter{theorem}{0}

In this section we put into evidence a bijection between 
$\cPchi (n)$ and $NC(n)$ which is implemented by the action of a
special permutation $\xtau_{ { }_{\chi} }$ of $\{ 1, \ldots , n \}$.
The main result of the section is Theorem \ref{thm:4.9}.  We will 
arrive to it by observing a certain construction of partition in 
$NC(n)$  -- the ``combined-standings partition'' associated to a 
pair $( \lambda , \chi ) \in \Luk (n) \times \{ \ell, r \}^n$,
which is introduced in Definition \ref{def:4.3}.

$\ $

\begin{definition}   \label{def:4.1}
Consider a pair $( \lambda , \chi )$ where $\lambda \in \Luk (n)$ 
and $\chi = (h_1, \ldots , h_n) \in \{ \ell , r \}^n$, and 
let $\pi \in \cP (n)$ be the output-time partition associated
to $( \lambda , \chi )$ in Definition \ref{def-and-rem:3.3}.
Let us record explicitly where are the occurrences of $\ell$ and 
of $r$ in $\chi$:
\begin{equation}   \label{eqn:41a}
\left\{   \begin{array}{ll}
\{ m \mid 1 \leq m \leq n, \, h_m = \ell \}
=: \{ m_{\ell} (1), \ldots , m_{\ell} (u) \} &  \mbox{with
                 $m_{\ell} (1) < \cdots < m_{\ell} (u)$,}       \\
                          &                                     \\
\{ m \mid 1 \leq m \leq n, \, h_m = r \}
=: \{ m_r (1), \ldots , m_r (v) \} &  \mbox{with
                              $m_r (1) < \cdots < m_r (v)$.}
\end{array}  \right.
\end{equation}

\vspace{6pt}

$1^o$ Suppose that in (\ref{eqn:41a}) we have $u \neq 0$.  We 
define a partition $\leftstand \in \cP (u)$ by the following 
prescription: two numbers $q,q' \in \{ 1, \ldots , u \}$ are 
in the same block of $\leftstand$ if and only if the numbers 
$m_{\ell} (q), m_{\ell} (q') \in \{ 1, \ldots , n \}$ belong 
to the same block of $\pi$.  The partition $\leftstand$ will 
be called the {\em left-standings partition} associated to 
$( \lambda , \chi )$. 

\vspace{6pt}

$2^o$ Likewise, if in (\ref{eqn:41a}) we have $v \neq 0$,  
then we define a partition $\rightstand \in \cP (v)$ via the 
prescription that $q,q' \in \{ 1, \ldots , v \}$ belong to the 
same block of $\rightstand$ if and only if $m_r (q), m_r (q')$ 
are in the same block of $\pi$.  The partition $\rightstand$ 
will be called the {\em right-standings partition} associated 
to $( \lambda , \chi )$. 
\end{definition}

$\ $

$\ $

\begin{remark-and-def}    \label{rem-and-def:4.2}
Consider the framework of Definition \ref{def:4.1}.  It will 
help the subsequent discussion if at this point we introduce 
some more terminology, which will also clarify the names chosen
above for the partitions $\leftstand$ and $\rightstand$.

\vspace{6pt}

$1^o$ Same as in Section 3, we will think of the numbers in 
$\{ 1, \ldots , n \}$ as of moments in time.  We will say that 
$t \in \{ 1, \ldots , n \}$ is a {\em left-time} (respectively a 
{\em right-time}) for $\chi$ to mean that $h_t = \ell$ (respectively
that $h_t = r$).  If $t$ is a left-time for $\chi$, then the unique 
$q \in \{ 1, \ldots , u \}$ such that $t = m_{\ell} (q)$ will be 
called the {\em left-standing of $t$ in $\chi$}.  Likewise, if $t$ 
is a right-time for $\chi$, then the unique 
$q \in \{ 1, \ldots , v \}$ such that $t = m_r (q)$ will be called 
the {\em right-standing of $t$ in $\chi$}.  

\vspace{6pt}

$2^o$ Let the rise-vector of $\lambda$ be
$\vec{\lambda} = ( p_1 -1, \ldots , p_n -1)$, with 
$p_1, \ldots , p_n \in \bN \cup \{ 0 \}$.  The numbers in the set
\begin{equation}  \label{eqn:42a}
I := \{ 1 \leq i \leq n \mid p_i > 0 \}
\end{equation}
will be called {\em insertion times} for $( \lambda , \chi )$.
Recall that the output-time partition $\pi$ associated to 
$( \lambda , \chi )$ has its blocks indexed by $I$; indeed, 
Equation (\ref{eqn:33c}) in Definition \ref{def-and-rem:3.3}
introduces this partition as 
\begin{equation}  \label{eqn:42}
\pi = \{ T_i \mid i \in I \}.
\end{equation}
With a slight abuse of notation, $\leftstand$ and $\rightstand$ 
from Definition \ref{def:4.1} can be written as
\begin{equation}  \label{eqn:42b}
\leftstand = \{ V_i \mid i \in I \} \ \ \mbox{ and } \ \ 
\rightstand = \{ W_i \mid i \in I \} 
\end{equation}
where for every $i \in I$ we put
\begin{equation}  \label{eqn:42c}
V_i := \{ 1 \leq q \leq u \mid m_{\ell} (q) \in T_i \}
\ \mbox{ and } \ 
W_i := \{ 1 \leq q \leq v \mid m_r (q) \in T_i \} .
\end{equation}
(Every $V_i$ is a block of $\leftstand$ unless $V_i = \emptyset$,
and every $W_i$ is a block of $\rightstand$ unless 
$W_i = \emptyset$.  Note that $V_i$ and $W_i$ cannot be empty at the
same time, since $| V_i | + |W_i| = |T_i| = p_i > 0$.)
\end{remark-and-def}

$\ $

$\ $

\begin{definition}   \label{def:4.3}
We continue to consider the framework of Definition \ref{def:4.1}
and of Notation \ref{rem-and-def:4.2}.  For every $i \in I$ let
us denote
\begin{equation}   \label{eqn:43a}
(n+1) - W_i := \{ n+1-q \mid q \in W_i \} 
\subseteq \{ u+1, \ldots , n \}.
\end{equation}
The partition 
\begin{equation}   \label{eqn:43b}
\cstand := \{ V_i \cup ( \, (n+1)-W_i \, ) \mid i \in I \}
\end{equation}
will be called the {\em combined-standings partition} associated
to $( \lambda , \chi )$.
\end{definition}

$\ $

$\ $

\begin{remark}   \label{rem:4.4}
The blocks of the partition $\cstand$ are indexed by the same 
set $I$ of insertion times that was used to index the blocks 
of the output-times partition $\pi = \{ T_i \mid i \in I \}$
in Notation \ref{rem-and-def:4.2}.  Moreover, we have
\[
| V_i \cup ( \, (n+1)-W_i \, ) |
= | V_i | \, + \, |W_i| = |T_i|, \ \ \forall \, i \in I;
\]
this shows that it must be possible to go between $\pi$ and 
$\cstand$ via the action of some suitably chosen permutation of 
$\{ 1, \ldots, n \}$.  We next make the easy yet significant
observation that the permutation in question can be picked so that
it only depends on $\chi$ (even though each of $\pi$ and $\cstand$ 
depends not only on $\chi$, but also on $\lambda$).
\end{remark}

$\ $

$\ $

\begin{definition}   \label{def:4.7}
Let $\chi$ be a tuple in $\{ \ell , r \}^n$.
We associate to $\chi$ a permutation $\xtau_{ { }_{\chi} }$ of 
$\{ 1, \ldots , n \}$ defined (in two-line notation for permutations)
as
\begin{equation}   \label{eqn:47a}
\xtau_{ { }_{\chi} } := \left(
\begin{array}{cccccc}
 1  & \cdots &  u  &  u+1  & \cdots &  n    \\
m_{\ell} (1) & \cdots & m_{\ell} (u) &  
    m_r (v)  & \cdots & m_r (1)
\end{array}  \right) ,
\end{equation}
where $m_{\ell} (1) < \cdots < m_{\ell} (u)$ and
$m_r (1) < \cdots < m_r (v)$ are as in Definition \ref{def:4.1} 
(the lists of occurrences of ``$\ell$'' and ``$r$'' in $\chi$).

In (\ref{eqn:47a}) we include the possibility that $v=0$ (when 
$u=n$ and $\xtau_{ { }_{\chi} }$ is the identity permutation), 
or that $u = 0$ (when $v=n$ and 
$\xtau_{ { }_{\chi} } (m) = n+1-m$ for every $1 \leq m \leq n$).
\end{definition}

$\ $

$\ $

\begin{lemma}  \label{lemma:4.8}
Consider a pair $( \lambda , \chi ) \in \Luk (n) \times 
\{ \ell, r \}^n$, and let $\pi \in \cP (n)$ be the output-time 
partition associated to $( \lambda , \chi )$ in Definition 
\ref{def-and-rem:3.3}.  We have 
\begin{equation}   \label{eqn:48a}
\xtau_{ { }_{\chi} } \cdot \cstand = \pi ,
\end{equation}
where $\cstand$ and $\xtau_{ { }_{\chi} }$ are as in Definitions 
\ref{def:4.3} and \ref{def:4.7}, respectively, and where the 
action of a permutation on a partition is as reviewed in 
Definition \ref{def:2.1}.3.
\end{lemma}

\begin{proof}  We use the notations established earlier in this 
section.  Clearly, (\ref{eqn:48a}) will follow if we prove that 
\begin{equation}   \label{eqn:48b}
\xtau_{ { }_{\chi} } 
\bigl( \, V_i \cup ( \, (n+1) - W_i \, ) \, \bigr)
= T_i, \ \ \forall \, i \in I.
\end{equation}

Let us fix an $i \in I$ for which we verify that (\ref{eqn:48b}) 
holds.  Since the sets $V_i \cup ( \, (n+1) - W_i \, ) \, )$ 
and $T_i$ have the same cardinality, it suffices to verify the 
inclusion ``$\subseteq$'' of the equality.  And indeed,
referring to how the permutation $\xtau_{ { }_{\chi} }$ is defined
in Equation (\ref{eqn:47a}), we have:
\[
q \in V_i  \ \Rightarrow \ 
\xtau_{ { }_{\chi} } (q) = m_{\ell} (q) \in T_i, \ \ \mbox{ and}
\]
\[
q \in (n+1)- W_i \ \Rightarrow \ 
\xtau_{ { }_{\chi} } (q) = m_r (n+1-q) \in T_i
\]
(where the fact that $m_r (n+1-q) \in T_i$ comes from Equation
(\ref{eqn:42c}), used for the element $n+1-q \in W_i$).  Thus both
$\xtau_{ { }_{\chi} } ( V_i )$ and 
$\xtau_{ { }_{\chi} } ( \, (n+1) - W_i \, )$ are subsets of $ T_i$,
and (\ref{eqn:48b}) follows.
\end{proof}  

$\ $

$\ $

\begin{example}    \label{example:4.11}
Consider (same as in Example \ref{example:3.4}.1) the concrete 
case when $n=5$, $\chi = ( r , \ell , \ell , r, \ell )$, and 
$\lambda \in \Luk (5)$ has rise-vector 
$\vec{\lambda} = ( 2, -1, 1, -1, -1 )$.  
As found in Example \ref{example:3.4}.1, the output-time partition 
associated to this $( \lambda , \chi )$ is 
$\pi = \{ \, \{ 1,2,4 \} , \, \{ 3,5 \} \, \}$.  
The set of insertion times for $( \lambda , \chi )$ of this example
is $I = \{ 1, 3 \}$; in order to illustrate the system of 
notation from Equation (\ref{eqn:42}), we then write $\pi$ as 
\[
\pi = \{ T_1, T_3 \}, \ \ \mbox{ with $T_1 = \{ 1,2,4 \}$
and $T_3 = \{ 3,5 \}$. }
\]
The left-times for $\chi$ are
$m_{\ell} (1) = 2, m_{\ell} (2) = 3, m_{\ell} (3) = 5$, 
and the right-times are $m_r (1) = 1, m_r (2) = 4$.  Since 
$m_{\ell} (1) \in T_1$ and $m_{\ell} (2), m_{\ell} (3) \in T_3$,
we get (in reference to the notations from Equations (\ref{eqn:42b})
and (\ref{eqn:42c})) that
\[
V_1 = \{ 1 \}, V_3 = \{ 2,3 \}, \mbox{ hence }
\leftstand = \{ \, \{ 1 \}, \, \{ 2,3 \} \, \} \in \cP (3).
\]
For the right-times we have $m_r (1), m_r (2) \in T_1$, giving
us that
\[
W_1 = \{ 1,2 \}, W_3 = \emptyset, \mbox{ hence }
\rightstand = \{ \, \{ 1,2 \} \, \} \in \cP (2).
\]
The combined-standings partition $\cstand$ associated to 
$( \lambda , \chi )$ has blocks
\[
V_1 \cup (6 - W_1) = \{ 1 \} \cup \{ 4,5 \} \ \mbox{ and }
\ V_3 \cup (6 - W_3) = \{ 2,3 \} \cup \emptyset,
\]
hence $\cstand = \{ \, \{ 1,4,5 \}, \, \{ 2,3 \} \, \}$.

Finally, the permutation associated to $\chi$ is
\[
\xtau_{ { }_{\chi} } =  \left(
\begin{array}{ccccc}
 1  &  2  &  3  &  4 & 5   \\
 2  &  3  &  5  &  4 & 1
\end{array}  \right) .
\]
As explained in the proof of Lemma \ref{lemma:4.8}, we have that
$\xtau_{ { }_{\chi} } ( \, \{ 1,4,5 \} \, ) = \{ 1,2,4 \} = T_1$
and $\xtau_{ { }_{\chi} } ( \, \{ 2,3 \} \, ) = \{ 3,5 \} = T_3$,
leading to the equality
$\xtau_{ { }_{\chi} } \cdot \cstand = \pi$.
\end{example}

$\ $

$\ $

Our next goal is to prove that the combined-standings partition
$\cstand$ always is a non-crossing partition.  In order to obtain
this, we first prove a lemma.

$\ $

\begin{lemma}  \label{lemma:4.5}
Consider the framework and notations of Definition \ref{def:4.3}.
Let us denote the maximal element of $I$ by $j$, and let us consider 
the block $S= V_j \cup \bigl( \, (n+1)-W_j \, \bigr)$ of the 
partition $\cstand$.  Then $S$ is an interval-block (i.e. 
$S= [t',t''] \cap \bZ$ for some $t' \leq t''$ in 
$\{1, \ldots , n \}$).
\end{lemma}

\begin{proof}  The conclusion of the lemma is clear if $|S| =1$,
so we will assume that $|S| \geq 2$, i.e. that $p_j \geq 2$.

The maximal insertion time $j$ considered in the lemma is either 
a left-time or a right-time for $\chi$.  We will write the proof by 
assuming that $j$ is a left-time (the case of a right-time is 
analogous).  We denote the left-standing of the time $j$ as $q$;
recall from Notation \ref{rem-and-def:4.2} that this amounts to 
$j = m_{\ell} (q)$.  

In view of the above assumptions, the deque-scenario associated to
$( \lambda , \chi )$ has the following feature: in the $j$-th
move of the deque device, the last $p_j$ balls of the input pipe 
(with labels between $1 + \sum_{i=1}^{j-1} p_i$ and $n$) are 
inserted into deque pipe from the left, and during the same move,
the ball $\ban$ goes into the output pipe.  Thus the configuration 
of balls residing in the deque-pipe at time $j$ is
\begin{equation}   \label{fig3}
\mbox{ 
\begin{tabular}{ccccccccc}
\hline
  & $\banminus$  &  $\banminustwo$  &  $\cdots$  & $\bass$ 
  & $\bax$ & $\cdots$ & $\bayy$  &     \\
\hline
\end{tabular} \ \ , }
\end{equation}
where $n' = n-1, n'' = n-2, \ldots , s = 1 + \sum_{i=1}^{j-1} p_i$, 
and where ``$\bax, \ldots ,\bayy$'' is the (possibly empty) 
configuration of balls that were in the deque pipe at time $j-1$.  
Let us also note here that each of the remaining moves of the 
deque device ($(j+1)$-th move up to $n$-th move) is either a 
left-$0$ move or a right-$0$ move, since the input pipe was
emptied at the $j$-th move.

Due to our assumption that $j$ is a left-time, it is certain that 
$V_j \neq \emptyset$ (we have in any case that $V_j \ni q$).  But
$W_j$ may be empty,  and we will discuss separately two cases.

\vspace{6pt}

{\em Case 1.}  $W_j = \emptyset$.

\noindent
In this case all the balls $\banminus , \ldots , \bass$ exit the
deque-pipe by its left side.  Some of the balls 
$\bax , \ldots , \bayy$ may also exit the deque-pipe by its left
side, but they can only do so after all of
$\banminus , \ldots , \bass$ are out of the way.  This immediately
implies that the times when $\banminus , \ldots , \bass$ exit the 
deque-pipe must have consecutive 
\footnote{ Note that the times themselves when 
$\banminus , \ldots , \bass$ exit the deque-pipe don't have to be 
consecutive, because they may 
be interspersed with 
some right-times used by balls from 
$\bax , \ldots , \bayy$.  The ``consecutive'' claim is only in 
reference to left-standings. }
left-standings.   It follows that in this case we have 
$S = V_j = \{ q, q+1, \ldots , q+p_j-1 \}$, and hence $S$ is
an interval-block of $\cstand$.

\vspace{6pt}

{\em Case 2.}  $W_j \neq \emptyset$.

\noindent
In this case some of the balls $\banminus , \ldots , \bass$ 
(at least one and at most $p_j -1$ of them) exit the deque-pipe 
by its right side.  We observe it is not possible to find
$s \leq a < b \leq n-1$ such that the ball $\baa$ exits the 
deque-pipe by its left side while $\bab$ exits by the right-side.
(Indeed, assume by contradiction that this would be the case.
In the picture
\[
\mbox{ 
\begin{tabular}{cccccccccccc}
\hline
  & $\banminus$ & $\cdots$ & $\bab$ & $\cdots$ & $\baa$ 
  & $\cdots$    & $\bass$  & $\bax$ & $\cdots$ & $\bayy$  &     \\
\hline
\end{tabular} \ \ , }
\]
one of the two balls $\baa$, $\bab$ must be the first to 
exit the deque-pipe -- but that's not possible, since the other 
ball will block it.)  As a consequence, there must exist a label
$c \in \{ s, \ldots , n-1 \}$ such that the balls 
$\bass , \ldots , \bac$ (i.e. the balls with labels in 
$[s,c] \cap \bZ$) exit the deque-pipe by the right side, while
the balls with labels in $(c, n-1] \cap \bZ$ (if any) exit by the 
left side.

We next observe that all the balls $\bax , \ldots , \bayy$ from 
the picture in (\ref{fig3}) must exit the deque-pipe by its right
side.  This follows via the same kind of ``blocking'' argument as
in the preceding paragraph.  (Say e.g. that $\bax$ wants to exit 
by the left -- then out of the two balls $\bass$ and $\bax$, none 
can be the first to exit  the deque-pipe, because it would be 
blocked by the other.)

Based on the above tallying of how the balls from the picture in 
(\ref{fig3}) exit the deque-pipe, a moment's thought shows that 
the set $W_j \subseteq \{ 1, \ldots , v \}$ must consist of the 

\noindent
$c-s+1$ largest numbers in $\{ 1, \ldots , v \}$ and that, likewise,
the set $V_j$ must be the sub-interval $\{ q, \ldots , u \}$ of 
$\{ 1, \ldots , u \}$.  Then $(n+1) - W_j$ comes to 
$\{ u+1, \ldots , u+ (c-s+1) \}$, and the union 
$S = V_j \cup ( \, (n+1)- W_j \, )$ is an interval-block of 
$\cstand$, as required.
\end{proof}

$\ $

$\ $

\begin{proposition}   \label{prop:4.6}
Let $n$ be a positive integer, and let $( \lambda , \chi )$ be a 
pair in $\Luk (n) \times \{ \ell, r \}^n$.  The combined-standings
partition $\cstand$ introduced in Definition \ref{def:4.3} is in 
$NC(n)$.
\end{proposition}

\begin{proof}  We proceed by induction on $n$.  The base case 
$n=1$ is clear, so we focus on the induction step: we fix an integer
$n \geq 2$, we assume the statement of the proposition holds for 
pairs in $\Luk (m) \times \{ \ell, r \}^m$ whenever 
$1 \leq m \leq n-1$, and we prove that it also holds for pairs in 
$\Luk (n) \times \{ \ell, r \}^n$.

Let us then fix a pair $( \lambda , \chi )$ in 
$\Luk (n) \times \{ \ell, r \}^n$, for which we will prove that 
$\cstand$ is in $NC(n)$.  We denote $\chi = ( h_1, \ldots , h_n )$,
and we denote the rise-vector of $\lambda$ as
$\vec{\lambda} =$

\noindent
$( p_1 -1, \ldots , p_n -1)$.
Besides $\cstand$, we will also work with 
the output-time partition $\pi \in \cP (n)$ associated to 
$( \lambda , \chi )$, and we will use the same notations as 
earlier in the section:
\[
\cstand = \{ V_i \cup (n+1) - W_i \mid i \in I \}
\mbox{ and }
\pi = \{ T_i \mid i \in I \} ,
\]
where $I = \{ 1 \leq i \leq n \mid p_i > 0 \}$, 
the set of insertion times for $( \lambda , \chi )$.
We will assume that $|I| \geq 2$ (if $|I| =1$ then clearly 
$\cstand = 1_n \in NC(n)$).  Same as in Lemma \ref{lemma:4.5},
we put $j := \max (I)$;  we thus have $p_j \geq 1$ and 
$p_{j+1} = \cdots = p_n = 0$. 

Let us put $m := n-p_j = \sum_{i=1}^{j-1} p_i$.  Then $m > 0$ 
(because the assumption $|I| \geq 2$ means there exists $i<j$ 
with $p_i > 0$), and also $m < n$ (since $p_j > 0$).  
We consider the $m$-tuple
\[
\chi_o := \chi \mid ( \, \{ 1, \ldots , n \} \setminus T_j \, )
\in \{ \ell , r \}^m
\]
(that is, $\chi_o = ( h_{t_1}, \ldots , h_{t_m} )$, where one
writes 
$\{ 1, \ldots , n \} \setminus T_j = \{ t_1, \ldots , t_m \}$
with $t_1 < \cdots < t_m$).  On the other hand, let us consider
the Lukasiewicz path $\lambda_o \in \Luk (m)$
determined by the requirement that
\[
\vec{\lambda_o} = 
( p_1 -1, \ldots , p_n -1) \mid 
( \, \{ 1, \ldots , n \} \setminus T_j \, ).
\]
It is easily seen that the combined-standings partition 
$\rho_{\lambda_o , \chi_o} \in \cP (m)$ associated to 
$( \lambda_o , \chi_o )$ is obtained from $\cstand \in \cP (n)$ 
by removing the block $V_j \cup ( \, (n+1) - W_j \, )$ of $\cstand$,
and then by re-naming the elements of the remaining blocks of 
$\cstand$ in increasing order.  (Indeed, for this verification all 
one needs to do is ignore the last group of $p_j$ balls which
moves through the pipes of the deque device, in the deque-scenario
determined by $( \lambda , \chi )$.)

Now, the block $V_j \cup ( \, (n+1) - W_j \, )$ removed out of 
$\cstand$ is an interval-block, by Lemma \ref{lemma:4.5}.  On the 
other hand, the partition $\rho_{\lambda_o , \chi_o}$ is 
in $NC(m)$, due to our induction hypothesis.  Thus the partition 
$\cstand \in \cP (n)$ is obtained via the insertion of an 
interval-block with $p_j (= n-m)$ elements into a partition from 
$NC(m)$.  This way of looking at $\cstand$ readily implies that 
$\cstand \in NC(n)$, and concludes the proof.
\end{proof}

$\ $

$\ $

It is now easy to prove the main result of this section, which 
is stated as follows.

$\ $

\begin{theorem}   \label{thm:4.9}
Let $\chi$ be a tuple in $\{ \ell, r \}^n$, and let the set of 
partitions $\cPchi (n) \subseteq \cP (n)$ be as in Definition 
\ref{def:3.5}.  Then $\cPchi (n)$ can also be obtained as
\begin{equation}   \label{eqn:49a}
\cPchi (n) = \bigl\{
\xtau_{ { }_{\chi} } \cdot \pi \mid \pi \in NC(n)
\bigr\} \subseteq \cP (n),
\end{equation}
with $\xtau_{ { }_{\chi} }$ as in Definition \ref{def:4.7}.
\end{theorem}

\begin{proof}  
We will show, equivalently, that 
$\bigl\{ \xtau_{ { }_{\chi} }^{-1} \cdot \pi 
 \mid \pi \in \cPchi (n) \} \ = \ NC(n)$.
Since on both sides of the latter equality we have sets of the 
same cardinality $C_n$, it suffices to verify the inclusion 
``$\subseteq$''.  But ``$\subseteq$'' is clear from Lemma 
\ref{lemma:4.8} and Proposition \ref{prop:4.6}, since for 
$\pi = \Phi_{\chi} ( \lambda )$ with $\lambda \in \Luk (n)$ we 
get $\xtau_{ { }_{\chi} }^{-1} \cdot \pi = \cstand \in NC(n)$.
\end{proof}

$\ $

$\ $

\begin{corollary}    \label{cor:4.10}
Let $n$ be a positive integer and let $\chi$ be a tuple in 
$\{ \ell , r \}^n$.

\vspace{6pt}

$1^o$ $\cPchi (n)$ contains the partitions $0_n$ and $1_n$ (from 
Notation \ref{def:2.1}.2), and also contains all the 
partitions $\pi \in \cP (n)$ which have $n-1$ blocks.

\vspace{6pt}

$2^o$ The bijection
$NC(n) \ni \pi \mapsto \xtau_{ { }_{\chi} } \cdot \pi
\in \cPchi (n)$ from Theorem \ref{thm:4.9} is a poset isomorphism, 
where on both $NC(n)$ and $\cPchi (n)$ we consider the partial 
order ``$\leq$'' defined by reverse refinement.  

\vspace{6pt}

$3^o$ $( \cPchi (n) , \leq )$ is a lattice.  The meet operation 
``$\wedge$'' of $\cPchi (n)$ is described via block-intersections 
-- the blocks of $\pi_1 \wedge \pi_2$ are non-empty intersections 
$V_1 \cap V_2$, with $V_1 \in \pi_1$ and $V_2 \in \pi_2$.
\end{corollary}

\begin{proof}
$1^o$ This follows from the fact that the set
$\{ 0_n , 1_n \} \cup 
\{ \pi \in \cP (n) \mid \pi \mbox{ has $n-1$ blocks} \}$
is contained in $NC(n)$ and is
sent into itself by the action of $\xtau_{ { }_{\chi} }$ (no 
matter what the permutation $\xtau_{ { }_{\chi} }$ is).

$2^o$ This is an immediate consequence of the observation that the 
partial order by reverse refinement is preserved by the action 
of either $\xtau_{ { }_{\chi} }$ or $\xtau_{ { }_{\chi} }^{-1}$.  

$3^o$ The fact that $( \cPchi (n) , \leq )$ is a lattice follows 
from $2^o$, since $( NC(n) , \leq )$ is a lattice.  The description 
of the meet operation of $\cPchi (n)$ holds because the meet 
operation of $NC(n)$ is given by block-intersections, and because 
the action of $\xtau_{ { }_{\chi} }$ on partitions respects 
block-intersections.
\end{proof}

$\ $

$\ $

\begin{remark}    \label{rem:4.11}
$1^o$ For every positive integer $n$, the permutations associated 
to the $( \ell , r )$-words $( \ell , \ldots , \ell )$ and 
$(r, \ldots , r )$ are 
\[
\xtau_{ { }_{ ( \ell , \ldots , \ell ) } } := \left(
\begin{array}{cccccc}
 1  &  2  &  \cdots  &  n    \\
 1  &  2  &  \cdots  &  n  
\end{array}  \right) , \ \ 
\xtau_{ { }_{ ( r , \ldots , r ) } } := \left(
\begin{array}{cccccc}
 1  &   2   &  \cdots  &  n    \\
 n  &  n-1  &  \cdots  &  1  
\end{array}  \right) .
\]
When plugged into Theorem \ref{thm:4.9}, this gives
$\cP^{ ( \ell, \ldots , \ell )} (n) =
\cP^{ ( r, \ldots , r )} (n) = NC(n)$, 
a fact that had already been noticed in Remark \ref{rem:3.7}.2.

\vspace{6pt}

$2^o$ Say that $n=4$ and that $\chi = ( \ell , r , \ell , r )$,
with associated partition
\[
\xtau_{ { }_{\chi} } =  \left(
\begin{array}{cccc}
 1  &  2  &  3  &  4    \\
 1  &  3  &  4  &  2
\end{array}  \right) .
\]
Theorem \ref{thm:4.9} gives, via an easy calculation,
that $\cPchi (4)$ contains all the 
partitions of $\{ 1,2,3,4 \}$ with the exception of 
$\{ \ \{ 1,4 \} \, , \, \{ 2,3 \} \ \}$ (in agreement with the
description of this particular $\cPchi (n)$ that was mentioned 
in section 1.2 of the introduction).
\end{remark}

$\ $

$\ $

In the sequel there will be instances when we will need to 
``read in reverse" a tuple from $\{ \ell , r \}^n$.  We conclude 
the section with an observation about that.

$\ $

\begin{definition}   \label{def:4.12}
For every $n \geq 1$ and 
$\chi = ( h_1, \ldots , h_n ) \in \{ \ell , r \}^n$, the tuple
\[
\chiopp := ( h_n, \ldots , h_1 )
\]
will be called the {\em opposite} of $\chi$.
\end{definition}

$\ $

\begin{proposition}   \label{prop:4.13}
Let $n$ be a positive integer, let $\chi$ be a tuple in 
$\{ \ell , r \}^n$, and consider the opposite tuple $\chiopp$.  
Then the sets of partitions $\cPchi (n)$ and 
$\cPchiopp (n)$ are related by the formula
\begin{equation}   \label{eqn:413a}
\cPchiopp (n) = \{ \piopp \mid \pi \in \cPchi (n) \} ,
\end{equation}
where the opposite $\piopp$ of a partition $\pi \in \cP (n)$ 
is as considered in Definition \ref{def:2.1}.4.
\end{proposition}

\begin{proof}  Let $\tau_o$ be the order-reversing 
permutation of $\{ 1, \ldots , n \}$ that was considered in 
Definition \ref{def:2.1}.4 ($\tau_o (m) = n+1-m$ for 
$1 \leq m \leq n$).  On the other hand let 
$u \in \{ 0, 1, \ldots , n \}$ be the number of occurrences of 
the letter $\ell$ in the word $\chi$, and let us consider the 
permutation
\begin{equation}   \label{eqn:413b}
\tau_u := \left(
\begin{array}{cccccccc}
1 &  2  & \cdots & u & u+1 & \cdots & n-1 &  n    \\
u & u-1 & \cdots & 1 &  n  & \cdots & u+2 & u+1 
\end{array}  \right).
\end{equation}
(Note that if $u$ happens to be $0$, then the permutation 
$\tau_0$ defined in (\ref{eqn:413b}) coincides, fortunately,
with the permutation $\tau_o$ that had been considered above.)

Let $\pi$ be a partition in $\cP (n) \setminus NC(n)$.  Let 
$V, W$ be two distinct blocks of $\pi$ which cross, and let 
$a < b < c < d$ be numbers such that $a,c \in V$ and $b,d \in W$.  
We leave it as an exercise to the reader to check via a 
case-by-case discussion that the numbers 
\[
\tau_u (a), \, \tau_u (b), \, \tau_u (c), \,
\tau_u (d) \in \{ 1, \ldots , n \}
\]
(despite not being necessarily in increasing order) 
ensure the existence of a crossing between the blocks 
$\tau_u (V)$ and $\tau_u (W)$ of the partition 
$\tau_u \cdot \pi \in \cP (n)$.

The argument in the preceding paragraph shows that 
$\{ \tau_u \cdot \pi \mid \pi \in \cP (n) \setminus NC(n) \}
\subseteq$
$\cP (n) \setminus NC(n)$.
A cardinality argument forces the latter inclusion to be an 
equality, and then from the fact that $\tau_u$ sends
$\cP (n)$ bijectively onto itself it also follows that we have
\begin{equation}   \label{eqn:413c}
\{ \tau_u \cdot \pi \mid \pi \in NC(n) \} \ = \ NC(n).
\end{equation}

Now let us consider the positions of the letters $\ell$ and $r$
in the words $\chi$ and $\chiopp$.  By tallying these positions 
and plugging them into the formulas for the permutations 
$\xtau_{ { }_{\chi} }$ and $\xtau_{ { }_{\chiopp} }$ (as in 
Definition \ref{def:4.7}), one immediately finds that 
\begin{equation}   \label{eqn:413d}
\xtau_{ { }_{\chiopp} } = 
\tau_o  \, \xtau_{ { }_{\chi} } \, \tau_u .
\end{equation}
So then we can write:
\begin{align*} 
\cPchiopp (n)
& = \{ \xtau_{ { }_{\chiopp} } \cdot \pi \mid \pi \in NC(n) \}      
    \ \ \mbox{ (by Theorem \ref{thm:4.9}) }                        \\
& = \{ \tau_o \xtau_{ { }_{\chi} } \tau_u \cdot \pi 
    \mid \pi \in NC(n) \}  \ \ \mbox{ (by Eqn.(\ref{eqn:413d}) }   \\
& = \{ \tau_o \xtau_{ { }_{\chi} } \cdot \pi '
    \mid \pi ' \in NC(n) \}  \ \ \mbox{ (by Eqn.(\ref{eqn:413c}) } \\
& = \{ \tau_o \cdot \pi '' \mid \pi '' \in \cPchi (n)) \}  
    \ \ \mbox{ (by Theorem \ref{thm:4.9}), } 
\end{align*}
and this establishes the required formula (\ref{eqn:413a}).
\end{proof}

\vspace{2cm}

\begin{center}
{\bf\large 5. \boldmath{$( \ell , r )$}-cumulant functionals}
\end{center}
\setcounter{section}{5}
\setcounter{equation}{0}
\setcounter{theorem}{0}

In this section we introduce the family of $( \ell , r )$-cumulant 
functionals associated to a noncommutative probability space.
In order to write in a more compressed way the summation formula 
defining these functionals, we first introduce a notation.

$\ $

\begin{notation}    \label{def:5.1}
{\em  [Restrictions of $n$-tuples.] }

\noindent
Let $\cX$ be a non-empty set, let $n$ be a positive integer,
and let $(x_1, \ldots , x_n)$ be an $n$-tuple in $\cX^n$.  For a 
subset $V = \{ i_1, \ldots , i_m \} \subseteq \{ 1, \ldots , n \}$,
with $1 \leq m \leq n$ and $1 \leq i_1 < \cdots < i_m \leq n$,
we will denote
\[
( x_1, \ldots , x_n ) \mid V :=
( x_{i_1}, \ldots , x_{i_m} ) \in \cX^m .
\]
The next definition uses this notation in two ways:

\vspace{6pt}

$\bullet$ for $\cX = \cA$ (algebra of noncommutative random 
variables);

\vspace{6pt}

$\bullet$ for $\cX = \{ \ell ,r \}$, when we talk about the 
restriction $\chi \mid V$ of a tuple $\chi \in \{ \ell , r \}^n$.
\end{notation}

$\ $

$\ $

\begin{prop-and-def}   \label{prop-and-def:5.2}
[$( \ell , r)$-cumulants.]

\noindent
Let $( \cA , \varphi )$ be a nocommutative probability space.  
There exists a family of multilinear functionals
\[
\Bigl( \,  \kappa_{\chi} : \cA^n \to \bC \, \Bigr)_{ n \geq 1,
                                \,  \chi \in \{ \ell , r \}^n } 
\]
which is uniquely determined by the requirement that 
\begin{equation}  \label{eqn:52a}
\left\{   \begin{array}{l}
\varphi (a_1 \cdots a_n ) =
\sum_{\pi \in \cPchi (n)}  \, \Bigl( \
\prod_{V \in \pi}
\kappa_{\chi \mid V} 
( \, (a_1, \ldots , a_n) \mid V \, ) \ \Bigr),   \\
                                                    \\
\mbox{ for every $n \geq 1$, $\chi \in \{ \ell , r \}^n$
       and $a_1, \ldots , a_n \in \cA$. }
\end{array}   \right.
\end{equation}
These $\kappa_{\chi}$'s will be called the 
{\em $( \ell , r )$-cumulant functionals} of $( \cA , \varphi )$.
\end{prop-and-def}

\begin{proof}  For $n=1$ we define 
$\kappa_{ ( \ell ) } = \kappa_{ (r) } = \varphi$.  We then 
proceed recursively, where for every $n \geq 2$, every 
$\chi \in \{ \ell , r \}^n$ and every $a_1, \ldots , a_n \in \cA$
we put 
\begin{equation}   \label{eqn:52b}
\kappa_{\chi} ( a_1, \ldots , a_n ) 
= \varphi ( a_1 \cdots a_n ) 
- \sum_{  \begin{array}{c} 
           {\scriptstyle \pi \in \cPchi (n)}  \\
           {\scriptstyle \pi \neq 1_n}        \end{array}  } 
\, \Bigl( \ \prod_{V \in \pi}
\kappa_{\chi \mid V} 
( \, (a_1, \ldots , a_n) \mid V \, ) \ \Bigr).
\end{equation}
It is immediate that (\ref{eqn:52b}) defines indeed a family of 
multilinear functionals which fulfil (\ref{eqn:52a}).  The 
uniqueness part of the proposition is also immediate, by 
following the (obligatory) recursion (\ref{eqn:52b}).
\end{proof}

$\ $

$\ $

\begin{remark}    \label{rem:5.3}
Let $( \cA , \varphi )$ be a noncommutative probability space,
let $( \kappa_n )_{n=1}^{\infty}$ be the family of free cumulant
functionals of $( \cA , \varphi )$, and let 
$\bigl( \, \kappa_{\chi} : \cA^n \to \bC \, \bigr)_{ n \geq 1,
\, \chi \in \{ \ell , r \}^n   }$
be the family of $( \ell ,r )$-cumulant functionals introduced
in Definition \ref{prop-and-def:5.2}.

\vspace{6pt}

$1^o$ As noticed in Remark \ref{rem:3.7}.2, one has
$\cP^{ ( \ell, \ldots , \ell )} (n) =
\cP^{ ( r, \ldots , r )} (n) = NC(n)$.  By plugging this fact
into the recursion (\ref{eqn:52b}) which characterizes the 
functionals $\kappa_{\chi}$, one immediately obtains the fact 
(already advertised in the introduction) that 
\[
\kappa_{ ( \, \underbrace{\ell, \ldots , \ell}_n \, ) }
= \kappa_{ ( \, \underbrace{r, \ldots , r}_n \, ) }
= \kappa_n, \ \ \forall \, n \geq 1.
\]

\vspace{6pt}

$2^o$ If $n \leq 3$, then we actually  have 
$\kappa_{\chi} = \kappa_n$ for every $\chi \in \{ \ell , r \}^n$.
This comes from the fact, observed in Remark \ref{rem:3.7}.3,
that $\cP^{ ( \chi ) } (n) = NC(n)$ when $n \leq 3$, no matter 
what $\chi \in \{ \ell , r \}^n$ we consider.

\vspace{6pt}

$3^o$ For $n \geq 4$, the functionals $\kappa_{\chi}$ with 
$\chi \in \{ \ell , r \}^n$ are generally different from 
$\kappa_n$.  Say for instance that 
$\chi = ( \ell , r , \ell , r) \in \{ \ell , r \}^4$, then 
the difference between the lattices 
$NC(4)$ and $\cP^{ ( \chi ) } (4)$ leads to the fact that 
for $a_1, \ldots , a_4 \in \cA$ we have
\[
\kappa_{ ( \ell ,r , \ell , r) } (a_1, \ldots , a_4)
\begin{array}[t]{cl}
=   &  \kappa_4 (a_1, \ldots , a_4)                     \\
    &   + \kappa_2 (a_1, a_4) \kappa_2 (a_2, a_3)
        - \kappa_2 (a_1, a_3) \kappa_2 (a_2, a_4).
\end{array}
\] 
\end{remark}

$\ $

\begin{remark}   \label{rem:5.4}
Let us also record here a formula, concerning 
$( \ell , r )$-cumulants, which is related to the reading of
$( \ell , r )$-words in reverse (i.e. to looking at $\chi$ 
versus $\chiopp$, as in Definition \ref{def:4.12} and Proposition 
\ref{prop:4.13}).  Suppose that $( \cA , \varphi )$ is a 
$*$-probability space.  Then, 

\noindent
with 
$\bigl( \,  \kappa_{\chi} : \cA^n \to \bC \, \bigr)_{ n \geq 1,
\,  \chi \in \{ \ell , r \}^n }$ denoting the family of 
$( \ell, r )$-cumulant functionals of $( \cA , \varphi )$,
one has 
\begin{equation}   \label{eqn:54a}
\left\{   \begin{array}{l}
\kappa_{\chi} ( a_1^{*}, \ldots , a_n^{*} ) 
= \overline{ \kappa_{\chiopp} (a_n, \ldots , a_1) },           \\
                                                               \\
\mbox{ $\ $ for every $n \geq 1$, $\chi \in \{ \ell , r \}^n$
            and $a_1, \ldots , a_n \in \cA$. }                 
\end{array}  \right.
\end{equation}
The verification of (\ref{eqn:54a}) is easily done by induction 
on $n$, where one relies on the bijections
\[ 
\cPchi (n) \ni \pi \mapsto \piopp \in \cPchiopp (n), \ \ 
\mbox{ for $n \geq 1$ and $\chi \in \{ \ell , r \}^n$, }
\]
that were observed in Proposition \ref{prop:4.13}.  (The proof 
of the induction step starts, of course, by writing that 
$\varphi (a_1^{*} \cdots a_n^{*}) =
\overline{ \varphi (a_n \cdots a_1) }$; each of the moments 
$\varphi (a_1^{*} \cdots a_n^{*})$ and $\varphi (a_n \cdots a_1)$ 
is then expanded into $( \ell , r )$-cumulants, in the way 
described in Definition \ref{prop-and-def:5.2}.)
\end{remark}

$\ $

$\ $

\begin{center}
{\bf\large 6.  \boldmath{$( \ell , r )$}-cumulants
of canonical operators}
\end{center}
\setcounter{section}{6}
\setcounter{equation}{0}
\setcounter{theorem}{0}

In this section we prove the theorem announced in section 1.4 of 
the Introduction.  We will adopt the framework and notations of 
the theorem -- so we are dealing with the $d$-tuples
$(A_1, \ldots , A_d)$ and $(B_1, \ldots , B_d)$ of left and 
respectively right canonical operators on $\cT$, which were 
defined in Equations (\ref{eqn:12c})--(\ref{eqn:12f}) of
section 1.2 by starting from two non-commutative polynomials
$f(z_1, \ldots , z_d)$ and $g(z_1, \ldots , z_d)$.  Recall that 
the coefficients of $z_{i_1} \cdots z_{i_n}$ in the polynomials
$f$ and $g$ are denoted as $\alpha_{(i_1, \ldots , i_n)}$ and 
as $\beta_{(i_1, \ldots , i_n)}$, respectively.  

In the formula claimed by the theorem we used the unified notation 
\begin{equation}   \label{eqn:13a}
A_i =: C_{i; \ell} \ \mbox{ and }
B_i =: C_{i; r}, \ \ \mbox{ for } 1 \leq i \leq d.
\end{equation}
In order to give a concise re-statement of that formula, let us 
also introduce a unified notation for the relevant coefficients 
$\alpha$ and $\beta$, as follows.

$\ $

\begin{definition}   \label{def:6.9}
{\em [Bi-words and bi-mixtures of coefficients.] }

\noindent
Let $n$ be a positive integer.

\vspace{6pt}

$1^o$ The elements of the set
$\{ 1, \ldots, d \}^n \times \{ \ell , r \}^n$ 
will be called {\em bi-words} of length $n$.  

\vspace{6pt}

$2^o$ Let $( \omega ; \chi )$ be a bi-word of length $n$, 
where $\omega = (i_1, \ldots , i_n ) \in \{ 1, \ldots , d \}^n$
and $\chi = (h_1, \ldots , h_n) \in \{ \ell , r \}^n$.
We denote  
\begin{equation}   \label{eqn:690a}
\gamma ( \omega ; \chi ) :=
\left\{   \begin{array}{ll}
\alpha_{ ( i_{m_r (v)}, \ldots , i_{m_r (1)},
           i_{m_{\ell} (1)}, \ldots , i_{m_{\ell} (u)} ) }, 
         & \mbox{ if $h_n = \ell$, }                   \\
         &                                             \\
\beta_{ ( i_{m_{\ell} (u)}, \ldots , i_{m_{\ell} (1)},
          i_{m_r (1)}, \ldots , i_{m_r (v)} ) },
         & \mbox{ if $h_n = r$, } 
\end{array}  \right.
\end{equation}
where $m_{\ell} (1) < \cdots < m_{\ell} (u)$ and
$m_r (1) < \cdots < m_r (v)$ record the lists of occurrences of 
$\ell$ and of $r$ in $\chi$ (same convention of notation as in 
Definition \ref{def:4.1}).  We will refer to 
$\gamma ( \omega ; \chi )$ as the {\em bi-mixture} of $\alpha$'s 
and $\beta$'s corresponding to the bi-word $( \omega ; \chi )$.
\end{definition} 

$\ $

The result we want to prove can then be stated as follows.

$\ $

\begin{theorem}   \label{thm:6.5}
For every $n \geq 1$ and every 
$\chi = ( h_1, \ldots , h_n) \in \{ \ell , r \}^n$, 
$\omega = (i_1, \ldots , i_n) \in \{ 1, \ldots , d \}^n$, 
one has
\begin{equation}  \label{eqn:69a}
\kappa_{\chi} ( C_{i_1; h_1}, \ldots , C_{i_n; h_n} )
= \gamma ( \omega ; \chi ).
\end{equation}
\end{theorem} 

$\ $

The remaining part of the section is devoted to the proof
of Theorem \ref{thm:6.5}.  The proof will go by formalizing, 
in Lemma \ref{lemma:6.7} below, the intuitive idea that the 
action of $A_1, \ldots , A_d$, 
$B_1, \ldots , B_d$ on the vacuum vector $\xivac \in \cT$ is 
closely related to the deque-scenarios from Section 3 of the 
paper.

In order to state Lemma \ref{lemma:6.7}, we need the concept 
(related to the one from Definition \ref{def:6.9}.2) of what is 
a ``reverse-bi-mixture'' of coefficients $\alpha$ and $\beta$.

$\ $

\begin{definition}    \label{def:6.6}
Let $n$ be a positive integer and let $( \omega ; \chi )$ be a 
bi-word of length $n$, where $\omega = ( i_1, \ldots , i_n )$ 
and $\chi = ( h_1, \ldots , h_n )$.  We will denote  
\begin{equation}   \label{eqn:66c}
\widetilde{\gamma} ( \omega ; \chi ) :=
\left\{   \begin{array}{ll}
\alpha_{ ( i_{m_r (1)}, \ldots , i_{m_r (v)},
           i_{m_{\ell} (u)}, \ldots , i_{m_{\ell} (1)} ) }, 
         & \mbox{ if $h_1 = \ell$  }                   \\
         &                                             \\
\beta_{ ( i_{m_{\ell} (1)}, \ldots , i_{m_{\ell} (u)},
          i_{m_r (v)}, \ldots , i_{m_r (1)} ) },
         & \mbox{ if $h_1 = r$, } 
\end{array}  \right.
\end{equation}
where $m_{\ell} (1) < \cdots < m_{\ell} (u)$ and
$m_r (1) < \cdots < m_r (v)$ record the lists of occurrences of 
$\ell$ and of $r$ in $\chi$ (same convention of notation as in 
Definition \ref{def:4.1} and in Definition \ref{def:6.9}).
We will refer to $\widetilde{\gamma} ( \omega ; \chi )$ as the
{\em reverse-bi-mixture} of $\alpha$'s and $\beta$'s 
corresponding to the bi-word $( \omega ; \chi )$.
\end{definition}

$\ $

\begin{remark}   \label{rem:6.4}
It is obvious that the reverse-bi-mixtures which were just 
introduced are related to the bi-mixtures from Definition 
\ref{def:6.9} by the formula
\begin{equation}   \label{eqn:690b}
\gamma ( \omega ; \chi ) 
= \widetilde{\gamma} ( \omegaopp ; \chiopp ),
\end{equation}
where for $\chi = ( h_1, \ldots , h_n)$ and
$\omega = (i_1, \ldots , i_n)$ we put
$\chiopp := ( h_n, \ldots , h_1 )$ (same as in 
Definition \ref{def:4.12}) and
$\omegaopp := ( i_n, \ldots , i_1 )$. 

Let us also record an immediate extension of Equation 
(\ref{eqn:690b}), namely that for every non-empty set 
$T \subseteq \{ 1, \ldots , n \}$ we have
\begin{equation}   \label{eqn:690c}
\gamma ( ( \omega ; \chi) \mid T ) 
= \widetilde{\gamma} ( \, ( \omegaopp ; \chiopp ) \mid (n+1)-T \, ),
\end{equation}
with $(n+1) - T := \{ n+1-t \mid t \in T \}$.  

\noindent
[The restrictions of bi-words that have appeared in 
Equation (\ref{eqn:690c}) are defined by the same convention 
as used in Notation \ref{def:5.1} -- e.g. we have
\[
( \omega ; \chi ) \mid T :=
( \omega \mid T \, ; \, \chi \mid T ) \in 
\{ 1, \ldots, d \}^m  \times  \{ \ell , r \}^m,
\]
where $m$ is the number of elements of $T$.]

In the statement of Lemma \ref{lemma:6.7} we will also use 
the following notation.
\end{remark}

$\ $

\begin{notation}    \label{def:6.2}
We denote
\begin{equation}   \label{eqn:62c}
\left\{   \begin{array}{l}
X_{0; \ell} = X_{0;r} = I \ \mbox{ (identity operator); }   \\
                                                            \\
X_{p; \ell} = \sum_{i_1, \ldots , i_p =1}^d
\alpha_{ (i_1, \ldots , i_p) }
L_{i_p} \cdots L_{i_1}, \ \mbox{ for $p \geq 1$;}           \\
                                                            \\
X_{p; r} = \sum_{i_1, \ldots , i_p =1}^d
\beta_{ (i_1, \ldots , i_p) }
R_{i_p} \cdots R_{i_1}, \ \mbox{ for $p \geq 1$.} 
\end{array}   \right.
\end{equation}   
The canonical operators $A_i, B_i$ that we are dealing with 
can then be written as 
\begin{equation}   \label{eqn:62d}
A_i = L_i^{*} \, \sum_{p=0}^{\infty} X_{p; \ell},
\ \ 
B_i = R_i^{*} \, \sum_{p=0}^{\infty} X_{p; r},
\ \mbox{ for $1 \leq i \leq d$.}
\end{equation}
(The sums in (\ref{eqn:62d}) are actually finite, since 
$X_{p; \ell} = X_{p;r} = 0$ for $p$ large enough.)

It will be convenient to use a ``unified left-right notation'' 
of the Equations (\ref{eqn:62d}), as follows.  We already have a
unified notation for $A_i$ and $B_i$ (the $C_{i;h}$ from 
Equation (\ref{eqn:13a})), and let us also denote
\[
L_i =: S_{i, \ell},  \ \ R_i =: S_{i,r}, 
\ \ \mbox{ for $1 \leq i \leq d$.}
\]
Then (\ref{eqn:62d}) can be put in the form
\begin{equation}   \label{eqn:62g}
C_{i;h} = S_{i;h}^{*} \, \sum_{p=0}^{\infty} X_{p; h},
\mbox{ for $1 \leq i \leq d$ and $h \in \{ \ell , r \}$. }
\end{equation}
\end{notation}

$\ $

\begin{lemma}  \label{lemma:6.7}
Let $n$ be a positive integer, and consider the following items:

\vspace{6pt}

$\bullet$ an $n$-tuple 
$\omega = (i_1, \ldots , i_n ) \in \{ 1, \ldots , d \}^n$;

\vspace{6pt}

$\bullet$ an $n$-tuple 
$\chi = (h_1, \ldots , h_n ) \in \{ \ell , r \}^n$;

\vspace{6pt}

$\bullet$ a Lukasiewicz path $\lambda \in \Luk (n)$ with rise-vector 
denoted as $\vec{\lambda} = (p_1 -1, \ldots , p_n - 1)$, where 
$p_1, \ldots , p_n \in \bN \cup \{ 0 \}$. 

\vspace{6pt}

\noindent
Let $\pi \in \cPchi (n)$ be the output-time partition associated
to $( \lambda , \chi )$ in Definition \ref{def-and-rem:3.3}.  Then 
we have
\begin{equation}  \label{eqn:67a}
X_{p_1;h_1}^{*} S_{i_1;h_1} \cdots X_{p_n;h_n}^{*} S_{i_n;h_n} 
\xivac = \overline{c} \, \xivac,  \ \ 
\mbox{ where }  \ \ 
c = \prod_{T \in \pi} \widetilde{\gamma}
( \, ( \omega ; \chi ) \mid T \, ).
\end{equation}
In Equation (\ref{eqn:67a}), the operators $X_{p;h}$ and $S_{i;h}$ 
are as in Notation \ref{def:6.2}, and the coefficients
$\widetilde{\gamma}$ are reverse-bi-mixtures, as in Definition
\ref{def:6.6}.
\end{lemma}

\begin{proof}  

Let $\{j_1<j_2<\ldots< j_t\}=\{i | p_i>0\}$ let 
$\pi=\{T_{j_1},\ldots, T_{j_t}\}$, where for $r=1,\ldots, t$, 
we have that $T_{j_r}$ denotes the block of the output-time 
partition corresponding to time $j_r$, i.e., the block whose 
minimal element is $j_r$.  We abbreviate $k=j_t$.  

We proceed by induction on $t$.  We first deal with the base case.  If $t=1$ then we must have $k=1$, $p_1=p_k=n$, $p_2=\ldots=p_n=0$, and $\pi=\{T_1\} = \left\{ \{1,\ldots,n\}\right\}$.  If we denote $\{m_{\ell}(1)<\ldots m_{\ell}(u)\} = \{i| h_i=\ell\}$ and $\{m_r(1)<\ldots<m_r(v)\}=\{i|h_i=r\}$ 
as in Definition \ref{def:6.6}, then we have
\begin{eqnarray*}
&\phantom{=}& \hskip -3em X_{p_1;h_1}^{*} S_{i_1;h_1} \cdots X_{p_n;h_n}^{*} S_{i_n;h_n} 
\xivac\\ 
&=& X_{n;h_1}^{*}\left( S_{i_1; h_1}\cdots S_{i_n;h_n}\xivac\right)\\
&=& X_{n;h_1}^{*} e_{m_{\ell}(1)}\otimes\ldots\otimes e_{m_{\ell}(u)}\otimes e_{m_r(v)}\otimes\ldots\otimes e_{m_r(1)} \\
&=& \overline{\widetilde{\gamma}\left((\omega,\chi)|T_1\right)} \xivac \\
&=& \overline{c}\xivac.
\end{eqnarray*}

Now assume that $t>1$ and that the conclusion of the lemma 
holds for all smaller values of $t$. Let 
$$f\colon\{1,2,\ldots, n-p_k\}\to\{1,\ldots, n\}\setminus T_k$$
denote the unique increasing bijection.  We abbreviate
\[
\widehat{\omega} = \left(i_{f(1)},\ldots, i_{f(n-p_k)}\right), 
\ \ 
\widehat{\chi} = \left(h_{f(1)},\ldots, h_{f(n-p_k)}\right),
\]
and we also denote 
$\vec{v} := \bigl( p_{f(1)}-1,\ldots, p_{f(n-p_k)-1} \bigr)$.
Let $\widehat{\lambda}$ be the Lukasiewicz path associated to 
$\vec{v}$, and let 
$\widehat{\pi}\in\mathcal{P}^{(\widehat{\chi})}(n-p_k)$ be the 
output-time partition associated to 
$(\widehat{\lambda},\widehat{\chi})$.  We now note, as is implicit 
in the discussions in Sections 3 and 4, that 
$$  f(\widehat{\pi}) = \{T_1,\ldots, T_{j_{t-1}}\}.  $$
Details of this observation are left to the reader.  Observe now 
that by the induction hypothesis we have
$$
X_{p_{f(1)},h_{f(1)}}^{*}S_{i_{f(1)};h_{f(1)}}\cdots X_{p_{f(n-p_k)},h_{f(n-p_k)}}^{*}S_{i_{f(n-p_k)};h_{f(n-p_k)}} =\overline{\widehat{c}}\xivac,
$$
where 
$$\widehat{c}= \prod_{T \in \widehat{\pi}} \widetilde{\gamma}
( \, ( \widehat{\omega} ; \widehat{\chi} ) \mid T \, )
 = \prod_{T\in\pi, T\not=T_k} \widetilde{\gamma}
( \, ( \omega ; \chi ) \mid T \, )
.$$
Let us list elements of the set $\{k,\ldots, n\}$ as $d_{k},\ldots, d_{n}$ by listing left elements first in the increasing order followed by the right elements in the decreasing order, i.e., the order in the list $d_{k},\ldots, d_n$ respects the order from the list $m_{\ell}(1),\ldots, m_{\ell}(u),m_r(v),\ldots, m_r(1)$.  Now note that we have
\begin{eqnarray*}
&\phantom{=}& \hskip -3em X_{p_k;h_k}^*S_{i_{k};h_k}\cdots X_{p_n;h_n}^*S_{i_n;h_n}\xivac\\
&=& X_{p_k;h_k}^* \left(S_{i_k;h_k}\cdots S_{i_n;h_n}\xivac\right) \\
&=& X_{p_k;h_k}^* e_{d_k}\otimes \ldots\otimes e_{d_n} \\
&=&\begin{cases} \overline{\alpha_{d_n,\ldots, d_k}}\xivac &,\mbox{ if }p_k=n-k+1\mbox{ and }h_k=\ell \\
\overline{\beta_{d_k,\ldots, d_n}}\xivac &,\mbox{ if }p_k=n-k+1\mbox{ and }h_k=r \\
 \overline{\alpha_{d_{k+p_k-1},\ldots, d_{k}}}e_{d_{k+p_k}}\otimes\ldots\otimes e_{d_{n}} &,\mbox{ if }p_k<n-k+1\mbox{ and }h_k=\ell\\
\overline{\beta_{d_{n+1-p_k},\ldots, d_n}}e_{d_{k}}\otimes\ldots\otimes e_{d_{n-p_k}} &,\mbox{ if }p_k<n-k+1\mbox{ and }h_k=r 
\end{cases}\\
&=& \overline{\widetilde{\gamma}\left((\omega,\chi)|T_k\right)} S_{i_{f(k)};d_{f(k)}}\cdots 
S_{i_{f(n-p_k)};d_{f(n-p_k)}}\xivac.
\end{eqnarray*}
Hence we have
\begin{eqnarray*}
&\phantom{=}& \hskip -3em X_{p_1;h_1}^*S_{i_1;h_1}\cdots  X_{p_{k-1};h_{k-1}}^*S_{i_{k-1};h_{k-1}}X_{p_k;h_k}^*S_{i_{k};h_k}\cdots X_{p_n;h_n}^*S_{i_n;h_n}\xivac \\
&=& X_{p_1;h_1}^*S_{i_1;h_1}\cdots  X_{p_{k-1};h_{k-1}}^*\left(\overline{\widetilde{\gamma}}\left((\omega,\chi)|T_k\right) S_{i_{f(k)};d_{f(k)}}\cdots 
S_{i_{f(n-p_k)};d_{f(n-p_k)}}\xivac\right) \\
&=& \overline{\widetilde{\gamma}\left((\omega,\chi)|T_k\right)} X_{p_1;h_1}^*S_{i_1;h_1}\cdots  X_{p_{k-1};h_{k-1}}^* S_{i_{f(k)};d_{f(k)}}\cdots 
S_{i_{f(n-p_k)};d_{f(n-p_k)}}\xivac \\
&=& \overline{\widetilde{\gamma}\left((\omega,\chi)|T_k\right)} X_{p_{f(1)};h_{f(1)}}^*S_{i_{f(1)};h_{f(1)}}\cdots X_{p_{f(n-p_k)};h_{f(n-p_k)}}^*S_{i_{f(n-p_k)};h_{f(n-p_k)}} \\
&=& \overline{\widetilde{\gamma}\left((\omega,\chi)|T_k\right)}\cdot \prod_{T\in\pi, T\not=T_k} \overline{\widetilde{\gamma}
( \, ( \omega ; \chi ) \mid T \, )}\xivac\\
&=&  \prod_{T\in\pi} \overline{\widetilde{\gamma}
( \, ( \omega ; \chi ) \mid T \,)} \xivac = \overline{c}\xivac.
\end{eqnarray*}
This concludes the induction step.
\end{proof}

$\ $

\begin{example}   \label{example:6.11}  
For clarity, let us follow the preceding lemma in the concrete 
case (also discussed earlier, in Examples \ref{example:3.4}.1 
and \ref{example:4.11}) where $n=5$, 
$\chi = ( r , \ell , \ell , r, \ell )$, and 
$\lambda \in \Luk (5)$ has rise-vector 
$\vec{\lambda} = ( 2, -1, 1, -1, -1 )$.  As found in Example 
\ref{example:3.4}.1, the output-time partition associated to 
this $( \lambda , \chi )$ is 
$\pi = \{ \, \{ 1,2,4 \} , \, \{ 3,5 \} \, \}$.  Let us also fix 
a tuple $\omega = (i_1, \ldots , i_5) \in \{ 1, \ldots , d \}^5$.
We have
$\widetilde{\gamma} 
\bigl( \, ( \omega ; \chi ) \mid \{ 1,2,4 \} \, \bigr)
= \widetilde{\gamma} 
\bigl( \, (i_1, i_2, i_4) ; (r , \ell , r ) \, \bigr)
= \beta_{(i_2, i_4, i_1)}$ and

\noindent
$\widetilde{\gamma} 
\bigl( \, ( \omega ; \chi ) \mid \{ 3,5 \} \, \bigr)
= \widetilde{\gamma} 
\bigl( \, (i_3, i_5) ; ( \ell , \ell ) \, \bigr)
= \alpha_{ (i_5, i_3) }$.
The constant $c$ from Equation (\ref{eqn:67a}) is thus 
$c = \beta_{(i_2, i_4, i_1)} \, \alpha_{ (i_5, i_3) }$,
and the formula claimed by the lemma should come to
\[
X_{3; r}^{*}    R_{i_1} X_{0; \ell}^{*} L_{i_2} 
X_{2; \ell}^{*} L_{i_3} X_{0; r}^{*}    R_{i_4}
X_{0; \ell}^{*} L_{i_5} \xivac = \overline{c} \, \xivac 
\]
for this particular value of $c$.  And indeed, let us record
how $\xivac$ travels when we apply to it the operators listed 
on the left-hand side of the above equation: we get
\begin{eqnarray*}
\xivac 
& \mapsto & L_{i_5}  \xivac = e_{i_5}                            \\
& \mapsto & R_{i_4} e_{i_5} = e_{i_5} \otimes e_{i_4}            \\
& \mapsto & X_{2; \ell}^{*} L_{i_3} (e_{i_5}\otimes e_{i_4}) 
  = X_{2;\ell}^{*} \, e_{i_3} \otimes e_{i_5} \otimes e_{i_4}  
  = \overline{ \alpha_{i_5, i_3} } \, e_{i_4}                    \\
& \mapsto & L_{i_2}( \overline{ \alpha_{i_5, i_3} } \, e_{i_4} ) 
  = \overline{ \alpha_{i_5, i_3} } \ e_{i_2} \otimes e_{i_4}     \\
& \mapsto & X_{3;r}^{*} R_{i_1}
  ( \overline{ \alpha_{i_5, i_3} } \, e_{i_2} \otimes e_{i_4}) 
  = \overline{ \alpha_{i_5, i_3} } \, X_{3;r}^{*} \, 
               (e_{i_2} \otimes e_{i_4} \otimes e_{i_1})
  = \overline{ \alpha_{i_5, i_3} } \cdot
    \overline{ \beta_{i_2,i_4,i_1} } \, \xivac ,
\end{eqnarray*}
as claimed.
\end{example}

$\ $

\begin{proposition}   \label{prop:6.10}
Let $n$ be a positive integer and let $( \omega ; \chi )$ be a 
bi-word of length $n$, where $\omega = ( i_1, \ldots , i_n )$ 
and $\chi = ( h_1, \ldots , h_n )$.  We have
\begin{equation}   \label{eqn:610a}
\phivac ( C_{i_1; h_1} \cdots C_{i_n; h_n} )
= \sum_{\pi \in \cPchi (n)}  \, \Bigl( \,
\prod_{T \in \pi}
\gamma ( \, ( \omega ; \chi ) \mid T \, ) \, \Bigr) .
\end{equation}
where the bi-mixtures ``$\gamma$'' on the right-hand side 
of the equation are as introduced in Definition \ref{def:6.9}.
\end{proposition}

\begin{proof}
Write each of $C_{i_1; h_1}, \ldots , C_{i_n; h_n}$ as a sum in 
the way indicated in Equation (\ref{eqn:62g}) of Notation
\ref{def:6.2}, then expand the ensuing product of sums; we get
\begin{equation}   \label{eqn:610b}
\phivac ( C_{i_1; h_1} \cdots C_{i_n; h_n} )
= \sum_{p_1, \ldots , p_n = 0}^{\infty}
\term_{ ( p_1, \ldots , p_n ) },
\end{equation}
where for every $p_1, \ldots , p_n \in \bN \cup \{ 0 \}$ 
we put
\begin{equation}   \label{eqn:610c}
\term_{ ( p_1, \ldots , p_n ) } 
\begin{array}[t]{rl}
:=  &  \phivac ( S_{i_1;h_1}^{*} X_{p_1;h_1} \cdots 
                 S_{i_n;h_n}^{*} X_{p_n;h_n} )            \\
    &                                                     \\
=   &  \langle  S_{i_1;h_1}^{*} X_{p_1;h_1} \cdots 
      S_{i_n;h_n}^{*} X_{p_n;h_n} \, \xivac \, , \, \xivac \rangle .
\end{array}
\end{equation}
We will proceed by examining what $n$-tuples
$( p_1, \ldots , p_n) \in ( \bN \cup \{ 0 \} )^n$ may contribute
a non-zero term in the sum from (\ref{eqn:610b}).  

\vspace{6pt}

So let $p_1, \ldots , p_n$ be in $\bN \cup \{ 0 \}$.  We make 
the following observations.

\noindent
$\bullet$  If there exists $m \in \{ 1, \ldots , n \}$ with 
$p_m + \cdots + p_n < (n+1) - m$, then 
$\term_{ ( p_1, \ldots , p_n ) } = 0$.  Indeed, if such an $m$
exists then it is immediately seen that 
\[
S_{i_m;h_m}^{*} X_{p_m;h_m} \cdots 
      S_{i_n;h_n}^{*} X_{p_n;h_n} \, \xivac  = 0,
\]
which makes the inner product from (\ref{eqn:610c}) vanish.

\noindent
$\bullet$  If $p_1 + \cdots + p_n > n$, then 
$\term_{ ( p_1, \ldots , p_n ) } = 0$.  Indeed, in this case 
the vector 

\noindent
$S_{i_1;h_1}^{*} X_{p_1;h_1} \cdots 
S_{i_n;h_n}^{*} X_{p_n;h_n} \, \xivac$
is seen to belong to the subspace 
\[
\mbox{span} \{ e_{j_1} \otimes \cdots \otimes e_{j_q} 
\mid 1 \leq j_1, \ldots , j_q \leq d \} \subseteq \cT,
\]
where $q = (p_1 + \cdots + p_n) - n > 0$.  The latter subspace is 
orthogonal to $\xivac$, and this again 
makes the inner product from (\ref{eqn:610c}) vanish.

\vspace{6pt}

The observations made in the preceding paragraph show that a 
necessary condition for $\term_{ ( p_1, \ldots , p_n ) } \neq 0$
is that 
\[
\left\{  \begin{array}{l}
p_m + \cdots + p_n \geq (n+1)-m, \ \ \forall \, 1 \leq m \leq n,  \\
                                                                  \\
\mbox{where for $m=1$ we must have } p_1 + \cdots + p_n = n.
\end{array}  \right.
\]
This says precisely that the tuple 
$( p_n - 1, \ldots , p_1 - 1 )$ is the rise-vector of a 
uniquely determined path $\lambda \in \Luk (n)$.  Hence the 
sum on the right-hand side of (\ref{eqn:610b}) is in fact, in 
a natural way, indexed by $\Luk (n)$.

\vspace{6pt}

Now let us fix a path $\lambda \in \Luk (n)$, where (consistent 
to the above) we denote the rise-vector of $\lambda$ as 
$\vec{\lambda} := ( p_n - 1, \ldots , p_1 - 1 )$.  If we put 
\[
\widetilde{p}_m := p_{n+1-m}, \
\widetilde{h}_m := h_{n+1-m}, \
\widetilde{i}_m := i_{n+1-m}, \ \ 1 \leq m \leq n,
\]
then Equation (\ref{eqn:610c}) can be re-written in the form
\[
\term_{ ( p_1, \ldots , p_n ) } 
= \langle  \xivac \, , \,
X_{\widetilde{p}_1; \widetilde{h}_1}^{*} 
S_{\widetilde{i}_1; \widetilde{h}_1} \cdots 
X_{\widetilde{p}_n; \widetilde{h}_n}^{*} 
S_{\widetilde{i}_n; \widetilde{h}_n} \, \xivac \, \rangle,
\]
where on the right-hand side we are in the position to invoke 
Lemma \ref{lemma:6.7}.  The lemma must be 
used in connection to the path $\lambda$ and the tuples 
$\chiopp = ( \widetilde{h}_1, \ldots , \widetilde{h}_n )$,
$\omegaopp = ( \widetilde{i}_1, \ldots , \widetilde{i}_n )$.
If we also denote 
\[
\widetilde{\pi} := \Phi_{\chiopp} ( \lambda ) 
\ \ \mbox{ (output-time partition associated to $\lambda$ 
and $\chiopp$), }
\]
the application of Lemma \ref{lemma:6.7} takes us to:
\[
\term_{ ( p_1, \ldots , p_n ) } = 
\prod_{ \widetilde{T} \in \widetilde{\pi} } \ 
\widetilde{\gamma} 
( \, ( \omegaopp ; \chiopp ) \mid \widetilde{T} ).
\]
Finally, we note that when $\widetilde{T}$ runs among the 
blocks of $\widetilde{\pi}$, the set 
$(n+1)- \widetilde{T}$ runs among the blocks of the opposite
partition $\widetilde{\pi}_{\mathrm{opp}}$.  Thus, in view of 
the relation between $\gamma$'s and $\widetilde{\gamma}$'s 
observed in Remark \ref{rem:6.4}, we arrive to the formula
\[
\term_{ ( p_1, \ldots , p_n ) } = 
\prod_{ T \in 
        ( \, \Phi_{\chiopp} ( \lambda ) \, )_{\mathrm{opp}} } 
\ \gamma ( \, ( \omega ; \chi ) \mid T \, ).
\]

The overall conclusion of the above discussion is that we have
\[
\phivac ( C_{i_1; h_1} \cdots C_{i_n; h_n} )
= \sum_{\lambda \in \Luk (n)} \
\prod_{ T \in 
        ( \, \Phi_{\chiopp} ( \lambda ) \, )_{\mathrm{opp}} } 
\gamma ( \, ( \omega ; \chi ) \mid T \, ).
\]
The only thing left to verify is, then, that the set of partitions
\[
\bigl\{ \, 
\bigl( \, \Phi_{\chiopp} ( \lambda ) \, \bigr)_{\mathrm{opp}}
\mid \lambda \in \Luk (n) \bigr\}
\]
coincides with $\cPchi (n)$.  But this is indeed true, since
$\{ \Phi_{\chiopp} ( \lambda ) \mid \lambda \in \Luk (n) \}
= \cPchiopp (n)$ (by the definition of $\cPchiopp (n)$), and in view 
of Proposition \ref{prop:4.13}.
\end{proof}

$\ $

\begin{ad-hoc-item}
{\bf Proof of Theorem \ref{thm:6.5}. }
We verify the required formula (\ref{eqn:69a}) 
by induction on $n$. 

For $n=1$ we only have to observe that
$\kappa_{ ( \ell ) } (A_i) = \gamma ( \, (i) ; ( \ell ) \, ), 
\ \ \forall \, 1 \leq i \leq d$
(both the above quantities are equal to $\alpha_{ (i) }$), and 
that
$\kappa_{ ( r ) } (B_i) = \gamma ( \, (i) ; ( r ) \, ), 
\ \ \forall \, 1 \leq i \leq d$
(both quantities equal to $\beta_{ (i) }$). 

Induction step: consider an $n \geq 2$, suppose the 
equality in (\ref{eqn:69a}) has already been verified for 
all bi-words of length $\leq n-1$, and let us fix a bi-word
$( \omega ; \chi )$ of length $n$, for which we want
to verify it as well.  Write explicitly
$\omega = (i_1, \ldots , i_n)$ and $\chi = (h_1, \ldots , h_n)$,
with $1 \leq i_1, \ldots , i_n \leq d$ and 
$h_1, \ldots , h_n \in \{ \ell , r \}$.  The joint moment
$\phivac ( C_{i_1;h_1} \cdots C_{i_n;h_n} )$
can be expressed as a sum over $\cPchi (n)$ in two ways: on 
the one hand we have it written as in Equation (\ref{eqn:610a}) 
of Proposition \ref{prop:6.10}, 
and on the other hand we can write it by using the 
moment$\leftrightarrow$cumulant formula (\ref{eqn:52a}) 
which was used to introduce the $( \ell , r )$-cumulants
in Definition \ref{prop-and-def:5.2}:  
\begin{equation}  \label{eqn:69b}
\phivac ( C_{i_1; h_1} \cdots C_{i_n; h_n} )
= \sum_{\pi \in \cPchi (n)}  \, \Bigl( \,
\prod_{V \in \pi}   \kappa_{\chi \mid V} 
( \, ( C_{i_1; h_1}, \ldots , C_{i_n; h_n} ) \mid V \, ) \ \Bigr) .
\end{equation}  
The induction hypothesis immediately gives us that, for every 
$\pi \neq 1_n$ in $\cPchi (n)$, the term indexed by $\pi$ 
in the two summations that were just mentioned (right-hand side
of (\ref{eqn:610a}) and right-hand side of (\ref{eqn:69b})) are 
equal to each other.  When we equate these two summations and 
cancel all the terms indexed by $\pi \neq 1_n$ in $\cPchi (n)$, 
we are left precisely with 
$\kappa_{\chi} ( C_{i_1; h_1}, \ldots , C_{i_n; h_n} )
= \gamma ( \omega ; \chi )$,
as required.
\hfill $\blacksquare$
\end{ad-hoc-item}

$\ $

$\ $

\noindent
{\bf\Large  Acknowledgements}

\vspace{4pt}

\noindent
This research work was started while the authors were participating
in the focus program on free probability at the Fields Institute in
Toronto, in July 2013.  The uplifting atmosphere and the support of
the Fields focus program are gratefully acknowledged.

\vspace{4pt}

\noindent
We also express our thanks to the anonymous referee who pointed 
to us the importance of re-writing the introduction in a way 
which better shows the motivation of the paper.

$\ $

$\ $

$\ $

$\ $

Mitja Mastnak

Department of Mathematics and Computing Science,

Saint Mary's University,

Halifax, Nova Scotia B3H 3C3, Canada.

Email: mmastnak@cs.smu.ca

$\ $

$\ $

Alexandru Nica

Department of Pure Mathematics, 

University of Waterloo,

Waterloo, Ontario N2L 3G1, Canada.

Email: anica@uwaterloo.ca


\begin{thebibliography}{99}

\bibitem{CNS2014} I. Charlesworth, B. Nelson, P. Skoufranis.
On two-faced families of non-commutative random variables.
Preprint, March 2014, available at arxiv.org/abs/1403.4907.

\vspace{10pt}

\bibitem{K1973} D.E. Knuth.
{\em The Art of Computer Programming, Volume 1: Fundamental 
Algorithms}, 2nd edition, Addison-Wesley, 1973.  

\vspace{10pt}

\bibitem{N1996} A. Nica. 
$R$-transforms of free joint distributions and non-crossing 
partitions, 
{\em Journal of Functional Analysis} 135 (1996), 271--296.

\vspace{10pt}

\bibitem{NS2006} A. Nica, R. Speicher.
{\em Lectures on the combinatorics of free probability},
London Mathematical Society Lecture Note Series 335,
Cambridge University Press, 2006.

\vspace{10pt}

\bibitem{S1994} R. Speicher. 
Multiplicative functions on the lattice of noncrossing partitions 
and free convolution,
{\em Mathematische Annalen} 298 (1994), 611--628. 

\vspace{10pt}

\bibitem{V1985} D. Voiculescu.
Symmetries of some reduced free product $C^{*}$-algebras, in
{\em Operator Algebras and Their Connections with Topology and
Ergodic Theory} (H. Araki, C.C. Moore, S. Stratila and 
D. Voiculescu, editors), 
Springer Lecture Notes in Mathematics Volume 1132,
Springer Verlag, 1985, pp. 556-588.

\vspace{10pt}

\bibitem{V1986} D. Voiculescu.
Addition of certain noncommuting random variables,
{\em Journal of Functional Analysis} 66 (1986), 323--346.

\vspace{10pt}

\bibitem{V2013a} D. Voiculescu.
Free probability for pairs of faces I,  
{\em Communications in Mathematical Physics} 332 (2014), 
955-980.                      % arxiv.org/abs/1306.6082.

\vspace{10pt}

\bibitem{V2013b} D. Voiculescu.
Free probability for pairs of faces II: 2-variables bi-free
partial R-transform and systems with rank $\leq 1$
commutation.  Preprint, August 2013, available at 
arxiv.org/abs/1308.2035.

\end{thebibliography}
\end{document}